\documentclass[12pt,leqno]{amsart}
\usepackage{amsmath,amssymb,amsthm}
\usepackage{eucal}

\newcommand\mylabel[1]{\label{#1}} 			

\newcommand\placefig[2]{\bigskip\begin{center} **********************************************************
\newline 
\emph{Place FIGURE #1 near EXAMPLE \ref{#2}.} 
\smallskip\newline
**********************************************************
\end{center}\bigskip}
\newcommand\placefigure[1]{\placefig{\ref{F#1}}{X#1}}
\renewcommand\placefig[2]{}

\newtheorem{lem}{Lemma}[section]
\newtheorem{cor}[lem]{Corollary}
\newtheorem{prop}[lem]{Proposition}
\newtheorem{thm}[lem]{Theorem}

\theoremstyle{remark}
\newtheorem*{defn}{Definition}                        
\newtheorem{exam}{Example}[section]
\newtheorem{resprob}{Research Problem}

\numberwithin{equation}{section}
\numberwithin{figure}{section}

\renewcommand{\phi}{\varphi}
\renewcommand{\epsilon}{\varepsilon}
\newcommand\eset{\varnothing}

\newcommand\bQ{\mathbf{Q}}

\newcommand\cA{\mathcal{A}}
\newcommand\cB{\mathcal{B}}
\newcommand\cC{\mathcal{C}}
\newcommand\cE{\mathcal{E}}
\newcommand\cG{\mathcal{G}}
\newcommand\cH{\mathcal{H}}

\newcommand\cL{\mathcal{L}}
\newcommand\cO{\mathcal{O}}
\newcommand\cP{\mathcal{P}}
\newcommand\cS{\mathcal{S}}
\newcommand\cT{\mathcal{T}}
\newcommand\cZ{\mathcal{Z}}

\newcommand\bbC{\mathbb{C}}
\newcommand\bbE{\mathbb{E}}
\newcommand\bbP{\mathbb{P}}
\newcommand\bbQ{\mathbb{Q}}
\newcommand\bbR{\mathbb{R}}
\newcommand\bbZ{\mathbb{Z}}	    

\newcommand\bgr[1]{\langle#1\rangle}
\newcommand\sgn{\operatorname{sgn}}
\newcommand\textb{\text{\rm b}}
\newcommand\Lat{\operatorname{Lat}}
\newcommand\Latb{\operatorname{Lat^{\textb}}}
\newcommand\Span{\operatorname{span}}
\newcommand\aff{\operatorname{aff}}
\newcommand\rk{\operatorname{rk}}
\newcommand\codim{\operatorname{codim}}
\newcommand\proj{\operatorname{proj}}

\newcommand\atinf{_{(\infty)}}
\newcommand\Pinf{\cP}

\newcommand\equivslistbegin{\begin{enumerate}\item[]\begin{enumerate}
\renewcommand{\theenumi}{}
\renewcommand{\theenumii}{\roman{enumii}}\renewcommand{\labelenumii}{(\theenumii)}}
\newcommand\equivslistend{\end{enumerate}\end{enumerate}}

\begin{document}


\title{Perpendicular Dissections of Space}
\author{Thomas Zaslavsky\\
	Binghamton University\\
	Binghamton, N.Y., U.S.A. 13902-6000}

\begin{abstract}
For each pair $(Q_i,Q_j)$ of reference points and each real number $r$ there is a unique hyperplane $h \perp Q_iQ_j$ such that $d(P,Q_i)^2 - d(P,Q_j)^2 = r$ for points $P$ in $h$.  Take $n$ reference points in $d$-space and for each pair $(Q_i,Q_j)$ a finite set of real numbers.  The corresponding perpendiculars form an arrangement of hyperplanes.  We explore the structure of the semilattice of intersections of the hyperplanes for generic reference points.  The main theorem is that there is a real, additive gain graph (this is a graph with an additive real number associated invertibly to each edge) whose set of balanced flats has the same structure as the intersection semilattice.  We examine the requirements for genericity, which are related to behavior at infinity but remain mysterious; also, variations in the construction rules for perpendiculars.  We investigate several particular arrangements with a view to finding the exact numbers of faces of each dimension.  The prototype, the arrangement of all perpendicular bisectors, was studied by Good and Tideman, motivated by a geometric voting theory.  Most of our particular examples are suggested by extensions of that theory in which voters exercise finer discrimination.  Throughout, we propose many research problems.
\end{abstract}

\thanks{This article is a redaction of a manuscript from 1984--85 reporting research performed originally in 1983 when I was at the Ohio State University and 1984--85 while I was a Visiting Scholar in Mathematical Research at the University of Evansville, with substantial additions (all of Sections \ref{matroids} and \ref{induced} and parts of others) and emendations in Year 2000; then further improved upon advice from a referee.  The research was supported by the National Science Foundation (through the Ohio State University) and the SGPNR.  My thanks to Clark Kimberling for the important reference \cite{Clem} and to Carol Nedlik for turning a remarkably messy old manuscript into a readable typescript that I could edit into presentability.}
\maketitle

\begin{center}
\emph{\bf Postpublication revision 8 May 2002: added reference to Voronoi in \S \ref{proj}: ``that goes back to the original paper \cite{Vor}''.}
\end{center}


\section{Introduction} \mylabel{intro}

Choose $n$ points $Q_1, \hdots, Q_n$ in $d$-dimensional Euclidean space and, for each pair
of points, take the hyperplane which is the perpendicular bisector of their connecting line segment.  Into how many regions does this arrangement of $\binom{n}{2}$ hyperplanes dissect the space?  In their article \cite{GT}, Good and Tideman showed that this number is, in general, equal to
\begin{equation}\mylabel{E1}
|s(n,n)| + |s(n,n-1)| + |s(n,n-2)| + \cdots + |s(n,n-d)|.
\end{equation}
Here $s(n,k)$ is the Stirling number of the first kind, one of whose many definitions is
that it equals $(-1)^{n-k}$ times the number of permutations having $k$ cycles of a set of
$n$ objects.

This geometry problem arose from a model of voter preference.  Suppose there are $n$ candidates, and $d$ issues on which each candidate has a position indicated by a real number.  A voter, who also has a (real number) position on each issue, prefers the nearer of two candidates.  How many different orderings of the candidates are 
possible?  Considering candidates $i$ and $j$, represented by points $Q_i$ and $Q_j$ in $d$-space, the perpendicular bisecting hyperplane of the segment $[Q_i,Q_j]$ divides those voters ranking $i$ over $j$ from those ranking $j$ over $i$.  The 
$\binom{n}{2}$ bisecting hyperplanes together dissect the space into regions such that all voters (i.e., points) in each region have the same ranking of candidates, while different regions yield different rankings.  Thus the model leads to the 
geometry problem described in the first paragraph and thus to the solution given by \eqref{E1}.  (This account ignores pseudo-orderings, in which the voter ranks some candidates equally.  Pseudo-orderings arise from voters located within bisecting 
hyperplanes; so to count all pseudo-orderings we should count all the cells, of all dimensions, into which the bisectors divide the space.  That was also done by Good and Tideman.)

Good and Tideman proved their formula by an induction on $n$ and $d$, in the course of which they proved (inductively) that for each $(d-k)$-dimensional flat of a 
special kind in the original arrangement the remaining hyperplanes form an arrangement of bisectors of $n-k$ points.  From this fact they deduced the number of $i$-dimensional faces of the original arrangement, for all $i$.  (Their proof for the special $(d-k)$-flats contains an oversight but their numerical result is correct.  We discuss this in Section \ref{induced}.)

We present a new proof and a generalization of Good and Tideman's formula \eqref{E1}, which explains the occurrence of the Stirling numbers.  Our approach is based on the fact that one can compute the number of regions of an arrangement $\cH$ of hyperplanes from a knowledge of the partially ordered set $\cL(\cH)$ of flats of intersection of the hyperplanes (ordered by reverse inclusion).

In broad outline: According to \cite[Theorem A]{FUTA}, $\cH$ has $|w_0|+|w_1|+\cdots+|w_d|$ regions, where the $w_i$ are the so-called ``Whitney numbers of the first kind'' of $\cL(\cH)$.  If $\cH$ is the arrangement of bisectors of $n$ points in $d$-space, then $\cL(\cH)$ is isomorphic to the set of all partitions of $n$ objects into at least $n-d$ parts, ordered by refinement.  (This was proved in effect by Good and Tideman.  We give a new and more general proof: see Section \ref{X1} for the result.)  Since the Whitney number $w_i$ of the partition lattice equals $s(n,n-i)$ (\cite{F-Rota}, \cite[\S 9]{FCT}), formula \eqref{E1} follows.  This proof sketch will be filled out below.

Our generalization is to allow other hyperplanes perpendicular to the lines $Q_iQ_j$ besides the bisectors.  We call these arrangements {\emph {Pythagorean arrangements of hyperplanes}}.  Then one must decide how to specify the location of the hyperplane, or, what is the same thing, of its foot $P_{ij}$ on $Q_iQ_j$.  Two obvious ways to do this are by specifying either the signed distance $d_{ij}(P_{ij})$ of the foot from the midpoint of the segment $[Q_i,Q_j]$ (positive toward $Q_j$, negative toward $Q_i$), or the proportional distance $d_{ij}(P_{ij}) / d(Q_i, Q_j)$.  The most appropriate way to locate $P_{ij}$, however, is neither of these.  We introduce the {\emph {Pythagorean coordinate}} of a point $P$ with respect to $Q_iQ_j$: it is the value
\begin{equation}\mylabel{E2}
\psi_{ij}(P) = d(P,Q_i)^2 - d(P,Q_j)^2 = 2\, d_{ij}(P) \cdot d(Q_i,Q_j).
\end{equation}
The significance of this coordinate is that, when it is employed to determine the locations of perpendicular hyperplanes, then for generic reference points $Q_1, \hdots, Q_n$ all concurrences of perpendiculars to lines $Q_iQ_j$ are determined in a nontrivial way by the Pythagorean coordinates of their feet (Section \ref{pyth}).\footnote{The only other specific use of Pythagorean coordinates of which I am aware is in Cacoullos \cite{Cac1}; they are his $\delta_{ij}(X)$.  He locates hyperplanes by proportional coordinates.  However, his problem is merely to find the nearest neighbor from amongst $d+1$ points and not the entire distance ranking of $n$ points.}  That enables us to compute the number of regions (and bounded regions and indeed flats and faces of any dimension) of the arrangement of perpendiculars.  We can also characterize (Section \ref{induced}) the arrangement induced in each flat by the original hyperplanes.  If on the other hand we locate the feet by signed distance from the midpoint, or proportional distance, or indeed any locating function of the form $c\psi_{ij}(P_{ij}) / d(Q_i, Q_j)^\alpha$ where $\alpha \neq 0$ and $c$ is any fixed nonzero multiplier, then (for generic reference points) there are no concurrences except those of bisectors (Section \ref{nonpyth}).  Thus the formula for the number of regions becomes less richly structured, depending not at all on the exact locations of the nonbisecting perpendiculars.

With the general theory we can treat more sophisticated voters, for instance those who prefer $Q_i$ to $Q_j$ only when the former is significantly closer than the latter---provided ``significantly closer'' is interpreted so as to be a linear condition.  The calculations, however, may become quite difficult.  Broadened examples like this are treated in Section \ref{enum}, with detailed formulas.  The natural, but nonlinear, interpretation of ``significantly closer'' to mean $d(V,Q_i) < d(V, Q_j) - \delta_{ij}$, where $\delta_{ij} > 0$ is fixed, leads to a totally different problem, unstudied but interesting and difficult: a quadric analog of hyperplane dissections (Section \ref{hyperbolic}).  

We reiterate that our results apply when the reference points $Q_i$ are chosen generically.  A definition of this concept is that, if $Q_1, \hdots, Q_n$ are generic, then shifting them slightly does not change the combinatorial type of the set of intersection flats of the associated arrangement of perpendiculars.  Exactly what genericity entails for the set of reference points is hard to say.  It does imply {\emph {simple}} position---that is, no $d+1$ of the $Q_i$ are affinely dependent---and more strongly, ``ideal general position'', which includes such projective properties as that no line $Q_iQ_j$ parallels a hyperplane determined by $d$ of the points.  However, these are not sufficient for genericity.  What else it may entail is insufficiently known despite our efforts in Section \ref{gp}.


\section{Arrangements of Hyperplanes} \mylabel{arrs}

We begin with some general theory of arrangements of hyperplanes from \cite{FUTA}.  An arrangement $\cH$ is a finite set{\footnote{In this paper all graphs and all sets of hyperplanes, flats, etc., are finite.}} of hyperplanes in Euclidean space ${\bbE}^d$ or real projective space $\bbP^d$,
together with the associated decomposition of the space into connected components, the {\emph {faces}} of $\cH$.  The $d$-dimensional faces are called {\emph {regions}}.  The number of $k$-dimensional faces ($k$-faces) is $f_k(\cH)$.  The number of bounded faces, in the Euclidean case, is $b_k(\cH)$.  
A {\emph {flat of $\cH$}} is any subspace obtained as the intersection of hyperplanes in $\cH$, excluding the null subspace in the Euclidean case.  The number of $k$-dimensional flats ($k$-flats) is $a_k(\cH)$.  The set $\cL(\cH)$ of all flats, when ordered by reverse inclusion, is a 
meet semilattice with $\hat 0$ = the whole space; in the projective case it is a lattice with $\hat 1 = \bigcap \cH$.  The rank function $\rk(x) = \codim x$ makes $\cL(\cH)$ a geometric semilattice (which is a geometric lattice with the interval over an atom deleted; see \cite{WW}) and in the projective case a geometric lattice.

Our primary interest is in Euclidean arrangements of hyperplanes perpendicular to lines, but it is easier to study them if we can also refer to their projective analogs.  Let $\cE$ be an arrangement of $N$ hyperplanes in $\bbE^d$.  The {\emph {projectivization}} $\cE_{\bbP}$ is the arrangement of $N+1$ hyperplanes in $\bbP^d$ consisting of the projective closure $h_\bbP$ of each hyperplane $h$ in $\cE$ and additionally the ideal hyperplane $h_\infty$.  The flats (and faces) of $\cE_{\bbP}$ are those of $\cE$ (actually, the projective closures $s_\bbP$ of all $s \in \cL(\cE)$) and the extra flats (and faces) in $h_\infty$, which reflect parallelisms among the Euclidean flats of $\cE$.

For a ranked partially ordered set $P$ with zero element we write $W_i$ for the number of rank $i$ elements (the {\emph {Whitney number of the second kind}}); thus $W_i(\cL(\cH)) = a_{d-i}(\cH)$.  More important for our purposes are the {\emph {Whitney numbers of the first kind}}, $w_i$. To define them we need the combinatorial M{\"o}bius function of $P$ (see \cite{FCT}), that is, $\mu : P \times P \rightarrow \bbZ$ defined recursively for increasing $y$ by 
\begin{equation*}
\begin{aligned}
\mu (x,y) &= 0 \quad&\text{ if } x \not\leq y, \\
\mu (x,x) &= 1, &\\
\mu (x,y) &= - { \sum_{z:z<y} } \mu (x,z) \quad&\text{ if } x < y .\\
\end{aligned}
\end{equation*}
Then 
$$
w_i = \sum \{ \mu(\hat0,y): \rk(y) = i \}.
$$
We take $P$ to be $\cL(\cH)$; hence 
$$
|w_i(\cL(\cH))| = (-1)^i w_i(\cL(\cH)) = \sum \{ |\mu(\hat 0,y)| : y \in \cL(\cH) \text{ and } \codim y = i \}
$$ 
by Rota's theorem \cite[\S 7]{FCT} that $(-1)^{\rk(\hat1)} \mu(\hat0,\hat1) > 0$ in a geometric lattice.  

With these preliminaries behind us, 
we can state the key facts of enumeration for arrangement of
hyperplanes \cite[Thms.\ A, B, and C]{FUTA}.  For a Euclidean arrangement $\cE$,
\begin{subequations} \mylabel{first-num-def}
\begin{equation}
f_d(\cE)= {\sum_{i=0}^d} | w_i(\cL(\cE))| \qquad\text{ and }
\end{equation}
\begin{equation}
b_d(\cE) = | {\sum_{i=0}^d} w_i(\cL(\cE)) |;
\end{equation}
\end{subequations}
the latter equals $(-1)^d\ {\sum_0^d} w_i(\cL(\cE))$ if $\cE$
has any $0$-flats.  For a nonvoid projective arrangement $\cA$,
\begin{subequations}
\begin{equation}
f_d(\cA) = {\tfrac 12} \sum_{i=0}^{d+1} |w_i(\cL(\cA))|,
\end{equation}
which if $\bigcap\cA = \eset$ is
\begin{equation}
   =  (-1)^k {\sum_{\substack{ i=0 \\ i\text{ even} }}^{d}} w_i(\cL(\cA)) .
\end{equation}
\end{subequations}
We can similarly count faces of each dimension.  Letting the {\emph {doubly indexed Whitney number of the first kind}} be $w_{ij}(P) = \sum \{ \mu(x,y) : \rk x  = i,\ \rk(y) = j \}$ (so, e.g., $w_{0i} = w_i$ and $w_{ii} = W_i$), we have:
\begin{subequations} \mylabel{face-num-def}
\begin{equation}
f_k(\cE) = {\sum_{j=d-k}^d} | w_{d-k,j}(\cL(\cE)) |,
\end{equation}
\begin{equation}
b_k(\cE) = | {\sum_{j=d-k}^d} w_{d-k,j}(\cL(\cE)) | ,
\end{equation}
\end{subequations}
and for a projective arrangement,
\begin{subequations} \mylabel{last-num-def}
\begin{equation}
f_k(\cA) = {\tfrac 12} {\sum_{j=d-k}^{d+1}} | w_{d-k,j}(\cL(\cA)) |
\end{equation}
provided $k > \codim(\bigcap \cA)$; if $\bigcap\cA = \eset$, this
\begin{equation}
    =  (-1)^k {\sum_{\substack{ j=d-k \\ d-j\text{ even} }}^d} |w_{d-k,j}(\cL(\cA))| .
\end{equation}
\end{subequations}
It is important in computations to remember that
$$
\sgn w_{ij} = (-1)^{j-i} 
$$
by Rota's theorem.

With formulas (\ref{first-num-def}--\ref{last-num-def}) in hand we can split 
the task of enumeration into two parts: determining the
structure of $\cL(\cH)$ from a suitable description, and evaluating its Whitney numbers.


\section{Perpendiculars specified by coordinates} \mylabel{perpcoord}

Here is the precise problem we want to solve:  We have a rule that assigns to any $n$-tuple $\bQ=(Q_1, \hdots, Q_n)$ of distinct {\emph {reference points}} in Euclidean $d$-space an arrangement of hyperplanes, 
$\cH = \bigcup_{i<j} \cH_{ij}$, where each $h_{ijk} \in \cH_{ij}$ is perpendicular to the line $Q_i Q_j$.  The rule consists of $\binom{n}{2}$ sets $R_{ij}$ of real numbers and an exponent $\alpha \in \bbR$.  Given an $n$-tuple $(Q_1, \hdots, Q_n)$, $\cH_{ij}$ consists of all those hyperplanes perpendicular to $Q_iQ_j$ whose {\emph {foot}} (the point of intersection with $Q_i Q_j$) is a point $P$ for which the quantity
\begin{equation} \mylabel{foot-num}
\psi_{ij}(P) d(Q_i, Q_j)^{-\alpha} \in R_{ij}.
\end{equation}
(In fact, since $\psi_{ij}(P) d(Q_i, Q_j)^{-\alpha}$ is the same for all points $P \in h$ if $h \perp Q_i Q_j$, we may call $\psi_{ij}(P)$ the {\emph {Pythagorean coordinate of $h$}} as a whole.)
We want to know the number of regions (and also faces and flats) of the arrangement $\cH$.  But since this obviously depends on the choice of $(Q_1, \hdots, Q_n)$, we are content to ask for the answer generically, in the sense discussed in the introduction.

(It is not necessary to assume $Q_i$ and $Q_j$ are distinct if there is no hyperplane $h_{ijk}$ specified perpendicular to $Q_iQ_j$.  However, if we allow reference points that are not distinct, we must define $Q_iQ_j$ to be $\aff(Q_i,Q_j)$, which is a point when $Q_i = Q_j$.)

Particular choices for $\alpha$ correspond to the three ways of specifying perpendiculars we mentioned in the introduction.  When $\alpha =1$, we are specifying twice the signed distance of the foot from the midpoint between $Q_i$ and $Q_j$.  When $\alpha = 2$, we are specifying twice the proportional distance.  When $\alpha = 0$, we specify the Pythagorean coordinate of the foot.  In each case, $R_{ij}$ specifies the values of the chosen coordinate at which we locate the feet of perpendiculars to $Q_iQ_j$.  We may think of $\alpha$ as universally fixed, with various 
choices of $\{ R_{ij}: 1 \leq i < j \leq n \}$ leading to various generic enumerative results.  If every $R_{ij} = \{ 0 \}$, then $\alpha$ is irrelevant and we have the original Good--Tideman arrangement of bisectors.  

The fact that there is a unique generic answer to our numerical question, and indeed that the structure of $\cL(\cH)$ itself can be determined generically, is significant.  In contrast it seems intuitively clear that, if we look at the isomorphism type of the whole arrangement $\cH$ (as a cell complex whose cells are the faces of $\cH$), there is no one generic type.  I expect that, when $n$ is large compared with $d$ (written $n \gg d$), there are many choices of $\bQ$ that are generic, in the sense of being deformable without altering the isomorphism type of $\cH$, and that yet yield mutually nonisomorphic arrangements.  At present this is unproved; we merely offer an example.

\begin{exam} \mylabel{Xnonisom} 
In Figure \ref{Fnonisom} we see two generic planar arrangements of all perpendicular bisectors from $4$ reference points.  Their face complexes are simplicial: all regions are triangular.  By comparing the bounded parts one can see that they are nonisomorphic.  (The reference points have different convexity types, or oriented matroids: in the first arrangement but not the second, one point is in the convex hull of the others.  This may be significant.)
\placefigure{nonisom}
\begin{figure} \mylabel{Fnonisom}
\vbox to 5truein{}
\begin{center}
\includegraphics{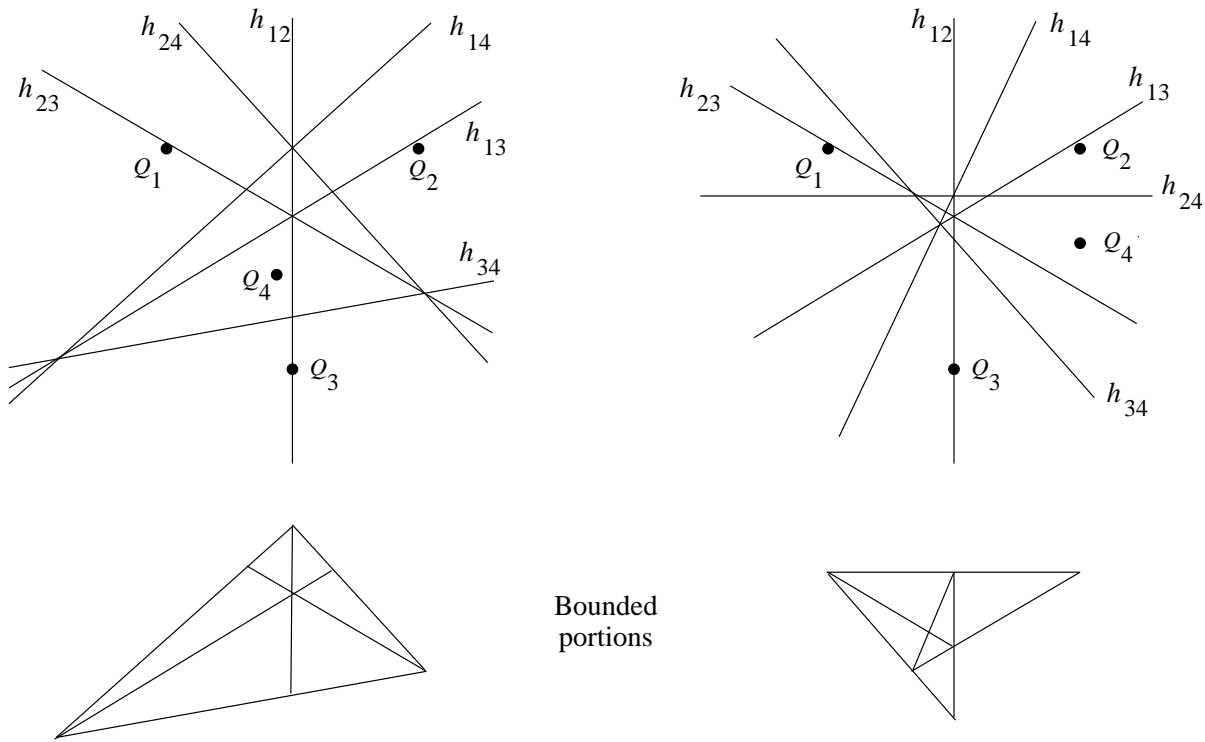}
\end{center}
\caption{Two generic but nonisomorphic arrangements of bisectors in the plane.}
\end{figure}

\end{exam}

\begin{resprob} \mylabel{RPface-types} 
Fix $d$, $\{ R_{ij} \}_{ij}$, and $\alpha$.  (a) Show that the face complex of $\cH$, the arrangement of perpendiculars determined by these data, is (with minor exceptions) not unique even when $\bQ$ is generic, if $n \gg d$.  (b) Estimate a lower bound, at least, for the number of distinct isomorphism types of generic face complexes.  (c) How large must $n$ be for the generic face complex to be nonunique?  (I suggest $n > d+1$.)

The results of Section \ref{nonpyth} suggest that the answers depend on whether $\alpha$ is 0 or not, but not otherwise on its value.
\end{resprob}

\begin{resprob} \mylabel{RPface-types-om} 
Is there a relationship between the structure of the face complex of $\cH$ and the oriented matroid structure of generic $Q_1, \hdots, Q_n$?
\end{resprob}

\begin{exam} \mylabel{Xdeform} 
\emph{(Affinographic arrangements and deformations of a Coxeter arrangement.)}  
A special case of particular importance is that in which $d=n$ and we choose an origin $O$ and reference points $Q_i$ so that $\vec{OQ_1},\cdots,\vec{OQ_n}$ are orthogonal vectors of length $\sqrt{2}$.  These vectors determine coordinates $(x_1,\hdots,x_n)$.  Let $P$ have coordinate vector $x$.  Then $\psi_{ij}(P) = \|x - \sqrt{2} e_i\|^2 - \|x-\sqrt{2}e_j\|^2 = x_j - x_i$ (where $e_i = \vec{OQ_i}/\sqrt{2}$, the unit basis vector), so the hyperplane with Pythagorean equation $\psi_{ij}(P) =\alpha$ has $x$-equation $x_j - x_i =\alpha$.  A Pythagorean arrangement is therefore a system of hyperplanes given by equations of the form $x_j - x_i =\alpha$ for $\alpha$ in some fixed set $R_{ij}$, specified for each $(i,j)$ with $0 < i < j \leq n$.  Such an arrangement we call \emph{affinographic}, since the hyperplanes are affine translates of those of the arrangement $\cA_n = \{x_j - x_i = 0 : 0 < i < j\leq n\}$, which represents the polygon matroid of the complete graph $K_n$.  Especially when the constant terms are integers, affinographic arrangements are known as {\it deformations of the Coxeter arrangement} $\cA_n$ \cite{Ath,P-S,Stan}.  Thus our results apply to deformations of $\cA_n$, and conversely, known characteristic polynomials of various deformations of $\cA_n$ can be applied to other Pythagorean arrangements as explained in Section \ref{invar} and illustrated in several examples of Section \ref{enum}.

Since all affinographic hyperplanes are orthogonal to $x_1 + \cdots + x_n = \sqrt{2}$, we can take the cross-section of an affinographic arrangement by the latter hyperplane.  This contains all the reference points, so it gives essentially the same arrangement in dimension $n-1$, though without the affinographic equations $x_j - x_i = \alpha$.
\end{exam}


\section{Perpendiculars via gain graphs} \mylabel{gg}

A more convenient way to locate perpendiculars is by a gain graph.  A {\emph {graph}} $\Gamma$ consists of a vertex set $V = V(\Gamma)$ and an edge set $E = E(\Gamma)$.  Multiple edges are permitted, indeed encouraged, but normally we allow only links---that is, edges with two distinct endpoints.  $V(e)$ denotes the set of endpoints of an edge $e$.  The {\emph{(connected) components}} of an edge set $S \subseteq E$ are the maximal subgraphs of $(V,S)$ that are connected by edges of $S$, including any isolated vertices, which are called {\emph {trivial components}}; $c(S)$ is the number of components of $S$.  A {\emph {real, additive gain graph}} $\Phi$ (sometimes, for brevity, called here just a ``gain graph'') consists of a graph $\Gamma = (V, E)$ and a {\emph {gain function}}, a mapping
\[
\phi : \big\{ (e; v_1,v_2) : e \in E,\ \{v_1,v_2\} = V(e) \big\} \rightarrow {\bbR}^+
\]
satisfying
\[
\phi(e; v_2, v_1) = -\phi(e; v_1, v_2).
\]
($\bbR^+$ is the additive group of real numbers.)  What this alternating property means is that the gain of $e$ from $v_2$ to $v_1$ is the inverse of that of $e$ from $v_1$ to $v_2$.  We define an edge set or subgraph to be {\emph {balanced}} if, for every circle (simple closed path) in it, the gains (taken in a consistent direction) sum to $0$.

There is a matroid theory of gain graphs extending that for ordinary graphs \cite[Part II]{BG}.  The complete lift matroid is implicit throughout our work here but we refer to it explicitly only in Sections \ref{matroids} (where we state a definition), \ref{induced}, and \ref{enum}.

We must show how gain graphs are related to arrangements of perpendiculars.

In one direction, suppose we start with reference points $\bQ = (Q_1, \hdots, Q_n)$ and an arrangement $\cH$ of perpendiculars based on $\bQ$.  We assume that each hyperplane in $\cH$ is associated with a specific reference line $Q_iQ_j$.  (The association need be explicit only if two such lines are parallel, which will never happen if $\bQ$ is generic.)  The {\emph {Pythagorean gain graph}} of $\cH$ (with respect to the given reference points), $\Psi_\bQ(\cH)$, has vertex set $\{ 1, 2, \hdots, n \}$ and an edge $e(h)$ between vertices $i$ and $j$ for each hyperplane $h$ associated with the reference line $Q_iQ_j$.  The gain of $e(h)$ is the Pythagorean coordinate of $h$:
\[
\psi(e(h);i,j) = \psi_\bQ(e(h);i,j) = \psi_{ij}(P),
\]
where $P$ is any point on $h$.

More important is the other direction.  We are given a real, additive gain graph $\Phi$.  The {\emph {Pythagorean hyperplane arrangement}} of $\Phi$ with reference points $Q_1, \hdots, Q_n$, written $\cH(\Phi; Q_1, \hdots, Q_n)$ or $\cH(\Phi; \bQ)$, is the arrangement of perpendiculars which has, for each edge $e \in E$, a hyperplane $h(e)$ perpendicular to $Q_iQ_j$, where $\{i,j\} = V(e)$, whose Pythagorean coordinate on $Q_iQ_j$ is $\phi(e;i,j)$.  Thus the Pythagorean gain graph of $\cH(\Phi; \bQ)$ is $\Phi$.  We think of $\Phi$ as a rule that specifies an arrangement of perpendiculars for each choice of reference points, equivalent to the real number sets $R_{ij}$ of the previous section but (as we shall see shortly in Section \ref{pyth}) more natural.

More generally we may wish to specify not the Pythagorean coordinates but some modification such as signed distance or proportional coordinates.  Given $\alpha \in \bbR$ and $\Phi$, the arrangement of perpendiculars $\cH(\alpha,\Phi;\bQ)$ is defined as $\cH(\Psi;\bQ)$ where $\Psi$ has gain function $\psi(e;i,j) = d(Q_i,Q_j)^\alpha \phi(e;i,j)$.  Thus the Pythagorean coordinate of a hyperplane $h(e)$ in $\cH(\alpha,\Phi;\bQ)$ is $d(Q_i,Q_j)^\alpha \phi(e;i,j)$ and $\Psi = \Psi_\bQ(\cH(\alpha,\Phi;\bQ))$, the Pythagorean gain graph.  We think of $(\alpha,\Phi)$ as a modified Pythagorean rule that determines an arrangement of perpendiculars on any given reference points.  Again $\Phi$ is equivalent to the real number sets of Section \ref{perpcoord}, but it is not normally the Pythagorean gain graph of the arrangement unless $\alpha = 0$.

In only one place we allow loops to appear, momentarily: when we contract, in Section \ref{induced}.  A loop's geometric interpretation depends on its gain.  A loop with gain $0$ corresponds to the whole space $\bbE^d$, which one might call the ``degenerate hyperplane''.  If $\bbE^d$ is among the ``hyperplanes'' of $\cH$, our convention is that there are no regions, since the complement of the hyperplanes is void; but there still are $d$-faces, since they are regions of the flat $s = \bbE^d$.  A loop with nonzero gain corresponds to the ideal hyperplane $h_\infty$, hence to nothing in $\bbE^d$.  For that reason we always discard such loops.

\begin{exam} \mylabel{Xpythgeneric}
Figure \ref{Fpythgeneric} shows a gain graph $\Phi$, four planar reference points, and the associated Pythagorean hyperplane arrangement $\cH(\Phi;\bQ)$.  We write $e_{ij}$ for an edge with endpoints $i$ and $j$ and, when necessary, $e_{ij}(\rho)$ to distinguish an edge $e_{ij}$ with gain $\rho = \phi(e_{ij};i,j)$.  Our $\cH(\Phi;\bQ)$ is generic because it has the generic intersection pattern required by Theorem \ref{T1}: since $\Phi$ has the single balanced circle $e_{12}e_{24}(3)e_{14}$, $h(e_{12}) \cap h(e_{24}(3)) \cap h(e_{14})$ is a multiple point of intersection; furthermore, generically there can be no other multiple points and no parallel flats except those implied by parallel edges like $e_{24}(0)$ and $e_{24}(3)$, which make $h(e_{24}(0)) \parallel h(e_{24}(3))$.  (The fact that $h(e_{24}(3))$ passes through $Q_3$ in this example is a coincidence that implies nothing about the hyperplane arrangement.)
\placefigure{pythgeneric}
\begin{figure} \mylabel{Fpythgeneric}
\vbox to 5.5truein{}
\begin{center}
\includegraphics{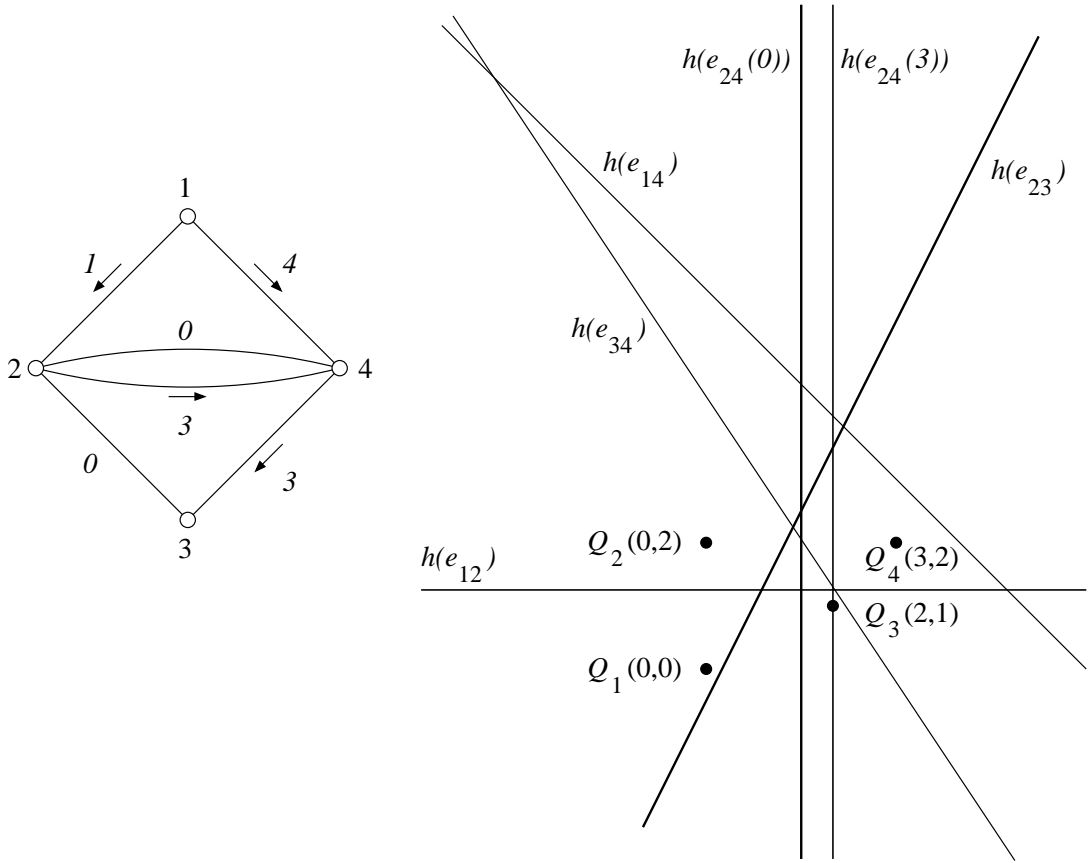}
\end{center}
\caption{A gain graph and an associated generic Pythagorean line arrangement.  (In the gain graph, the number by an edge is its gain.  The arrow indicates the direction for reading the gain, so reversing direction negates the gain.  The arrow may be omitted if the gain is $0$.)}
\end{figure}

The foot $P$ of a hyperplane $h(e_{ij})$ is positioned so that $\psi_{ij}(P) = \phi(e_{ij};i,j)$.  In practice it may be easier to locate $P$ by its signed distance from the midpoint $M_{ij}$ of $Q_iQ_j$: this is $d(M_{ij},P)$, taken as positive toward $Q_j$ and negative toward $Q_i$.  The formula for signed distance from a Pythagorean gain graph $\Phi$ is
$$
d(M_{ij},P) = \tfrac12 \phi(e_{ij};i,j)/d_{ij},
$$
where $d_{ij} = d(Q_i,Q_j)$.
\end{exam}\smallskip

\begin{exam} \mylabel{Xsdgeneric} 
In Figure \ref{Fsdgeneric} is a generic non-Pythagorean arrangement $\cH(1,\Phi;\bQ)$.  It is described by a gain graph $\Phi$ interpreted as specifying twice the signed distance of $h(e)$; that is, $d(M_{ij},P) = \frac12 \phi(e;i,j)$.  (See the previous example for notation.)  In the Pythagorean gain graph $\Psi = \Psi_\bQ$ of $\cH(1,\Phi;\bQ)$, the gains are $\psi(e;i,j) = d_{ij}\phi(e;i,j)$.  The only concurrence of lines is that of the three bisectors, which correspond to the balanced triangle in $\Psi$ made up of the edges with gain $0$.
\placefigure{sdgeneric}
\begin{figure} \mylabel{Fsdgeneric}
\vbox to 5truein{}
\begin{center}
\includegraphics{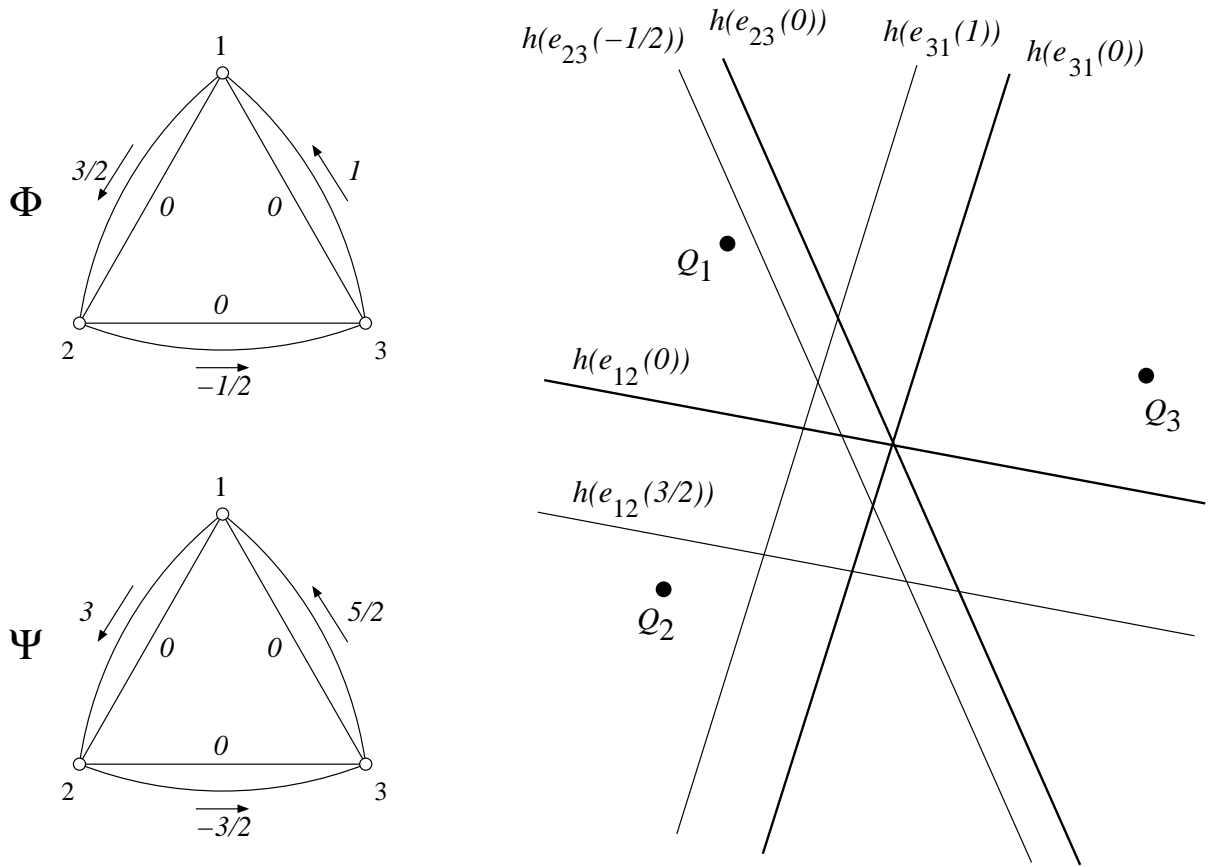}
\end{center}
\caption{A gain graph $\Phi$; an associated generic non-Pythagorean arrangement $\cH(\alpha,\Phi;\bQ)$ having $n=3$, $d=2$, and $\alpha = 1$; and the Pythagorean gain graph $\Psi=\Psi_\bQ$ of the arrangement.  The distances among the reference points are $d_{12} = 2$, $d_{23} = 3$, and $d_{13} = \frac52$.}
\end{figure}
\end{exam}


\section{Pythagorean perpendiculars} \mylabel{pyth}

Here is our central result.

\begin{thm} \mylabel{T1} 
Let $n > d \geq 1$, let $\Phi$ be a real, additive gain graph on vertex set $\{ 1, \hdots, n \}$, and for distinct points $Q_1, \hdots, Q_n \in {\bbE}^d$ let $\cH = \cH(\Phi; Q_1, \hdots, Q_n)$ be the arrangement of perpendiculars based on $Q_1, \hdots, Q_n$ with $\Phi$ as its Pythagorean gain graph.  Suppose $Q_1, \hdots, Q_n$ are generic.  Then the intersection of a subset $\cS \subseteq \cH$ corresponding to an edge set $S$ of $\Phi$ with $n-m$ connected components is void if $S$ is unbalanced or if $m>d$; otherwise it is a nonempty flat of dimension $d-m$.  

Furthermore, two subsets $\cS_1$ and $\cS_2$ have the same nonvoid intersection if and only if $S_1 \cup S_2$ is balanced and the connected components of $S_1$ and $S_2$ each partition the vertices (into at least $n-d$ parts) in the same way.
\end{thm}

Note that whether $Q_1, \hdots, Q_n$ are generic depends on the particular $\Phi$.  What we assert is that for each $\Phi$ a generic choice exists and has certain describable intersection properties.

The proof depends on several lemmas.  The first two show, by demonstrating the close connection of balanced circles in $\Psi_\bQ(\cH)$ to nonvoid hyperplane intersections, exactly why gain graphs are such a natural way to describe Pythagorean arrangements.

In the proof we assume that $Q_i$ and $Q_j$ are distinct if $i$ and $j$ are adjacent in $\Phi$, but we assume nothing if $i$ and $j$ are nonadjacent.

\begin{lem} \mylabel{L2a} 
Suppose $C$ is a circle in $\Phi$ and $\cC$ is the corresponding set of hyperplanes in $\cH(\Phi;\bQ)$.  If $C$ is unbalanced, $\bigcap \cC = \eset$.  If $C$ is balanced and $e \in C$, then 
$\bigcap \cC = \bigcap \big( \cC \setminus \{h(e)\} \big)$.
\end{lem}

\begin{proof} 
Take $C = e_{12}e_{23}\cdots e_{l,l+1}e_{l+1,1}$ and $e = e_{l+1,1}$.  Let $s = \bigcap (\cC\backslash \{h(e)\})$.  For any point $P \in s$,
\begin{equation}\mylabel{Epsam}
\psi_{l+1,1}(P) = -[\psi_{12}(P) + \cdots + \psi_{l,l+1}(P)]
\end{equation}
by \eqref{E2}.  The numbers $\psi_{i,i+1}(P) = \psi(e_{i,i+1})$ for $i = 1,2,\hdots,l$ by definition of $h(e_{i,i+1})$.  If $C$ is balanced, then by \eqref{Epsam} all of $s$ lies in $h(e_{l+1,1})$.  If not, no point of $s$ is in $h(e_{l+1,1})$.
\end{proof}

\begin{lem} \mylabel{L2}
In the situation of Lemma \ref{L2a}, suppose $Q_1, \ldots, Q_{l+1}$ are affinely independent.  If $C$ is balanced, then $\bigcap \cC$ is a $(d-l)$-flat.\footnote{Lemmas \ref{L2a} and \ref{L2} in the planar case of three vertices of a triangle with one perpendicular's foot on each extended edge of the triangle have been known to triangle geometers.  See for instance \cite{Clem} and \cite[\S 1.5, Exer.\ 7]{Cox1}.  Although \cite{Clem} is the earliest source I know, in Clark Kimberling's opinion (personal information, 1984) the planar theorem is many decades older; that is not hard to believe since it has the appearance of an easy exercise.}
\end{lem}

\begin{proof} 
Because $Q_1,\ldots,Q_{l+1}$ are affinely independent, $l$ hyperplanes perpendicular to $Q_1Q_2,$ $\ldots,Q_lQ_{l+1}$ are linearly independent.  Thus their intersection is $(d-l)$-dimensional.  Now the lemma follows from Lemma \ref{L2a}.
\end{proof}

\begin{lem} \mylabel{L3}
Suppose $\bigcap \cH(\Phi;\bQ)$ has nonempty intersection $s$.  Then
\begin{enumerate}
\item[(a)] $\Phi$ is balanced, and
\item[(b)] $s$ is the intersection of a subset \,$\cS$ of\, $d - \dim s$ hyperplanes.  Any such subset corresponds to a forest $S$ in $\Phi$.
\end{enumerate}
\end{lem}

\begin{proof}
For part (a), consider an unbalanced circle $C$ in $\Phi$.  By Lemma \ref{L2a}, the hyperplanes corresponding to the edges of $C$ have empty intersection.

For part (b), $\cS$ exists, of size $\codim s$ but no less, by the modular law of dimension in $\bbP^d$ applied to $s_\bbP$.  By Lemma \ref{L2a}, the minimality of $\cS$ prevents $S$ from containing a balanced circle.
\end{proof}

Let $p_{ij}$ denote the ideal point on $Q_iQ_j$, that is $(Q_iQ_j)\atinf$, where $s\atinf$ means $s_\bbP \cap h_\infty$.  (If $Q_i = Q_j$, then $p_{ij} = \eset$.  For the meaning of this, see Section \ref{gp}.)  Let $s\rightarrow s^*$ denote the natural polarity of ideal subspaces, viz.\ $p_{ij}^* = (h_{ij})\atinf$ for any hyperplane $h_{ij} \perp Q_iQ_j$ and, for a general subspace $s$ of $h_\infty$, $s^* = t\atinf$ where $t$ is any orthogonal complement of any affine space $s'$ whose ideal part $s'\atinf$ equals $s$.  We can regard the points $p_{ij}$ as the edges of a graph on the vertex set $\{ 1, \hdots, n\}$, identified with $\{ Q_1, \hdots, Q_n \}$.  Let $\Pinf$ denote the set of all $p_{ij}$.

\begin {lem} \mylabel{L4} 
Given an $n$-tuple $\bQ = (Q_1, \hdots, Q_n)$ of distinct points in $\bbE^d$ and a set $T \subseteq \Pinf$, we have 
\begin{equation}\mylabel{E3}
\dim T \leq \min(n-c(T)-1, d-1)
\end{equation}
The class of $n$-tuples $\bQ$ such that equality holds for all $T$ has the property of genericity; that is, it is open and dense in $(\bbE^d)^n$.  It excludes all nonsimple $n$-tuples $\bQ$.
\end{lem}

\begin{proof}
Let $\cG$ be the set of $\bQ$ having equality in \eqref{E3}.  

Consider a component $T_k$ of $T$ with vertex set $\{ i_1, \hdots, i_r\}$.
It is easy to see that $T_k$ spans the space 
$t_k = \aff(Q_{i_1}, \hdots, Q_{i_r})\atinf$, whose dimension is $\leq r-2$.  Summing over all components yields \eqref{E3}.

If $\bQ$ is not simple, say $Q_1, Q_2, \hdots, Q_r$ are a minimal affinely dependent
subset with $r \leq d + 1$.  Let $T = \{ p_{12}, p_{23}, \hdots, p_{r-1,r} \}$.
Then $\dim T= \dim \aff(Q_1, \hdots, Q_r) - 1 = r - 3$, but $n-c(T) = r-1$.  So 
$\bQ \notin \cG$.

We wish to prove that $\cG$ is open and dense.  For an $n$-tuple $\bQ = (Q_1, \hdots, Q_n)$, let $D(\bQ) = \{ T \subseteq \Pinf : T \text{ has strict inequality in \eqref{E3} and } n - c(T) \leq d \}$.  
If $D(\bQ)$ is empty, then \eqref{E3} holds with equality for all $T \subseteq \Pinf$.  Moving any points of $\bQ$ very slightly will not reduce the dimension of any set $T$; thus $\cG$ is open; furthermore, for an arbitrary $\bQ$, $D(\bQ') \subseteq D(\bQ)$ if $\bQ'$ is very near $\bQ$.

Suppose now that $\bQ \notin \cG$.  If $\bQ$ is not simple, an
arbitrarily slight deformation of its points $Q_1, \hdots, Q_n$ will make
it so.  Thus we can assume simplicity.  So, if $T_k$ is a component of $T$
with $n-c(T) \leq d$, and $T_k$ has vertices $i_1, \hdots, i_r$, then 
$t_k = \aff(Q_{i_1}, \hdots, Q_{i_r})\atinf$ has dimension $r-2$.  If $T$ is a minimal element of $D(\bQ)$ with nontrivial components $T_1, \hdots, T_m$, then 
$$
\dim(t_1 \cup \dots \cup t_{m-1}) = {\sum^{m-1}_{1}} \dim t_k
$$
while
$$
\dim(t_1 \cup \dots \cup t_m) < {\sum^{m}_{1}} \dim t_k 
= \dim(t_1 \cup \dots \cup t_{m-1}) + \dim(t_m).
$$
But $t_m$ is determined by points $Q_{i_1}, \hdots, Q_{i_r}$ which have no
effect on $t_1 \cup \cdots \cup t_{m-1}$.  Thus by varying these points
(arbitrarily slightly) to positions $Q'_{i_1}, \hdots, Q'_{i_r}$ we can
make the dimension of $(t_1 \cup \cdots \cup t_{m-1}) \cup t_m$ equal
to its maximum value of ${\sum^m_1} \dim t_k$.  The new $n$-tuple
$\bQ'$ will have $D(\bQ') \subset D(\bQ)$.  Continuing in this way we can find
$\bQ'' \in \cG$ arbitrarily near to $\bQ$. Thus $\cG$ is
dense.
\end{proof}

If $\bQ = (Q_1, \hdots, Q_n)$ satisfies \eqref{E3} with equality, we say $\bQ$, or
$Q_1, \hdots, Q_n$, have {\emph {ideal general position}}.  This is a
strong kind of affine general position, ruling out all unnecessary
parallelisms.  Lemma \ref{L4} allows us to restrict our consideration of the
general form of perpendicular arrangements described by a fixed $\Phi$ to
those $\bQ$ in ideal general position.

We shall need the dual formulation.  Recall that $\cS_\bbP$ contains $h_\infty$ by definition.

\begin{lem} \mylabel{L4a}
Let $\cS \subseteq \cH(\Phi;\bQ)$, corresponding to $S \subseteq E(\Phi)$.  Then
\begin{equation*}
\dim(\bigcap \cS_\bbP) \geq \max(d-1-[n-c(S)], -1),
\end{equation*}
with equality if $\bQ$ has ideal general position.
\end{lem}

\begin{proof}
Dulaize Lemma \ref{L4}.
\end{proof}

\begin{lem} \mylabel{L4b}
Let $\cS \subseteq \cH(\Phi;\bQ)$ correspond to a forest $S \subseteq E(\Phi)$ with $m$ edges where $m \leq d$.  Assume $\bQ$ has ideal general position.  Then $\dim(\bigcap \cS) = d-m$.
\end{lem}

\begin{proof}
Let $s = \bigcap \cS$, $s' = \bigcap \{ h_\bbP : h \in \cS \}$, and $s'' = \bigcap \cS_\bbP$.  We know $\dim s'' = d-m-1$ since $c(S) = n-m$.

The proof is by induction on $m$.  The base cases $m = 0,\ 1$ are trivial.  Thus let $1 < m \leq d$.  Choose $e \in \cS$ and set $T = S \setminus \{e\}$, $\cT = \cS \setminus \{h(e)\}$, $t = \bigcap \cT$, and $t'' = \bigcap \cT_{\,\bbP}$.  By induction, $\dim t = d-m+1$.  Thus $\dim t'' = d-m$.

Now, either $s=t$, or $\eset \neq s \subset t$ and $\dim s = \dim t - 1$, or $s = \eset$.  The last case is impossible, for $s = \eset \implies s' \subseteq h_\infty \implies s'' = s'$, whence $\dim s' = d-m-1$, contradictiing the fact that $\dim s' \geq \dim t - 1$ by the modular law.  The first case is impossible because $s = t \implies s' = t' \implies s'' = t''$, but $\dim s'' \neq \dim t''$.  Therefore $\eset \neq s \subset t$ and $\dim s = \dim t - 1$.
\end{proof}

\begin{lem} \mylabel{L4c}
Let $\cS \subseteq \cH(\Phi;\bQ)$ correspond to a forest $S \subseteq E(\Phi)$ with $m$ edges where $m > d$.  Assume $\bQ$ has ideal general position.  Then $\dim(\bigcap \cS)$ is a point or void, and for generic $\bQ$ it is void.
\end{lem}

\begin{proof}
That $\bigcap \cS$ is contained in a point follows directly from Lemma \ref{L4b}.

Let $\cG$ be the set of all $\bQ \in (\bbE^d)^n$ that have ideal general position and for which $\bigcap \cS = \eset$.  We need to show that $\cG$ is open and dense in $(\bbE^d)^n$.

If $\bigcap \cS = \eset$, shifting $\bQ$ slightly will leave it so.    Thus $\cG$ is open.

Suppose $\bigcap \cS$ is a point $P$.  We show how to deform $\bQ$ to make $\bigcap \cS$ empty.  Let $i$ be an end vertex in $S$, $e$ its incident edge, $j$ its neighbor, and $T = S \setminus \{e\}$.  Then $\bigcap \cT = \{P\}$ by induction if $m > d+1$ or Lemma \ref{L4b} if $m = d+1$.  The hyperplane $h(e)$ contains $P$; thus
\begin{equation*}
d(P,Q_i)^2 = d(P,Q_j)^2 + \phi(e;i,j).
\end{equation*}
If we shift $Q_j$ slightly off the cylinder described by this
equation, $\bigcap \cS$ becomes empty.  This proves that $\cG$ is dense.
\end{proof}

\begin{exam} \mylabel{Xnongeneric}
In Lemma \ref{L4c} it is indeed possible that $\bigcap \cS$ be a point.  Here is an example in the plane.  Take a point $P$ and three lines through $P$ making small angles, say $l_1, l_2, l_3$ with $l_2$ between the other two lines.  Choose a foot $P_2$ on $l_2$ at some distance from $P$ and erect a short perpendicular segment on $P_2$, bisected by $l_2$ and not intersecting either other line.  Call its endpoints $Q_2$ (nearer to $l_1$) and $Q_3$ (nearer to $l_3$).  From $Q_2$ drop a perpendicular to $l_1$ and continue it slightly past $l_1$ to a point $Q_1$.  Similarly, drop a perpendicular to $l_3$ from $Q_3$ and continue it to $Q_4$.  The reference points $Q_1, Q_2, Q_3, Q_4$ are plainly in ideal general position; the perpendiculars $l_1$, $l_2$, and $l_3$ to $Q_1Q_2$, $Q_2Q_3$, and $Q_3Q_4$ are described by their Pythagorean gain graph $\Psi$ which is a tree of $3$ edges; yet $l_1, l_2, l_3$ are concurrent.
\placefigure{nongeneric}
\begin{figure} \mylabel{Fnongeneric}
\vbox to 3.1truein{}
\begin{center}
\includegraphics{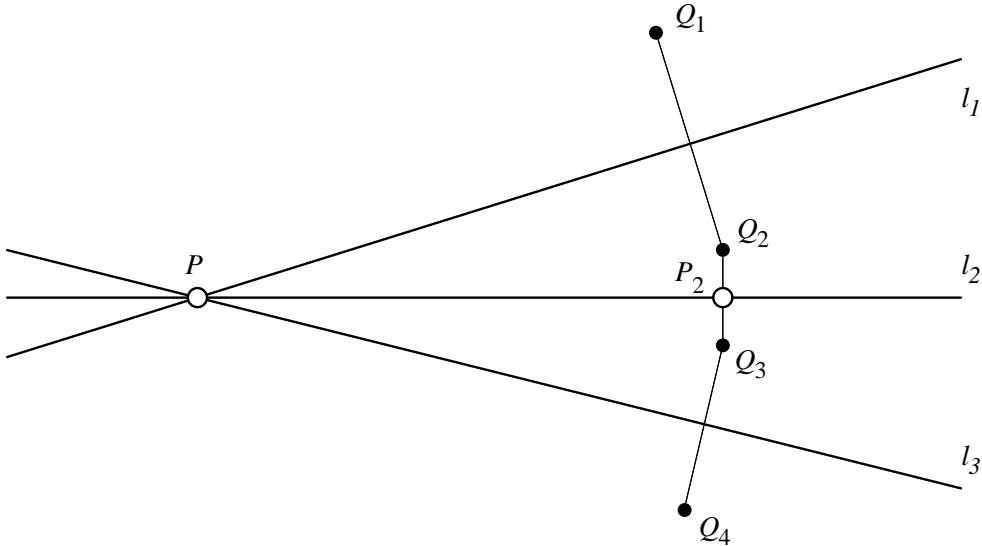}
\end{center}
\caption{A nongeneric arrangement of lines whose reference points nevertheless have ideal general position.}
\end{figure}

\end{exam}

Examples like this show that Lemma \ref{L4c} is where ideal general position ceases to be sufficient for genericity of $\bQ$.

\begin{lem} \mylabel{L5}
Let $\Phi$ be a fixed Pythagorean gain graph.  Let $\bQ = (Q_1, \hdots, Q_n)$ have ideal general position in $\bbE^d$, where $n > d \geq 1$.  Take a subset $\cS \subseteq \cH(\Phi;\bQ)$, corresponding to an edge set $S$ of $\Phi$.  Then $\bigcap \cS$ is a nonempty subspace of dimension $d-(n-c(S))$ if $S$ is balanced and has $c(S) \geq n-d$, it is a point or void (and it is void for generic $\bQ$) if $S$ is balanced and $c(S) < n-d$, and it is void if $S$ is unbalanced.
\end{lem}

\begin{proof} 
Suppose $S$ contains a balanced circle.  By Lemma \ref{L2a} we can remove an edge of $C$ from $S$ without changing either the intersection of the hyperplanes or the number of components of $S$.  So we may as well assume $S$ contains no balanced circle.

If $S$ contains an unbalanced circle, then $\bigcap \cS = \eset$ by Lemma \ref{L3}.

The remaining case is that of a forest, which is treated in Lemma \ref{L4c}.
\end{proof}

\begin{proof}[Proof of Theorem \ref{T1}]
For each $S \subseteq E(\Phi)$ let $\cG(S)$ be the set of $n$-tuples $\bQ$ for which $\cS$ has intersection as described in Lemma \ref{L5}.  There are finitely many of these sets $\cG(S)$, each open and dense in $(\bbE^d)^n$; thus their intersection  $\cG$ is open and dense as well.  Since every $\bQ \in \cG$ has the intersection properties in the first part of the theorem, these properties are generic, as claimed.

As for the second half of the theorem, where $\bigcap {\cS}_1 = \bigcap {\cS}_2 \neq \eset$, this makes both intersections equal to $\bigcap ({\cS}_1 \cup {\cS}_2)$.  Thus $S_1 \cup S_2$ is balanced, $n - c(S_1 \cup S_2) \leq d$, and 
$c(S_1 \cup S_2) = c(S_1) = c(S_2)$.  The latter entails that $S_1$ and $S_2$ induce the same partition of the vertex set.
\end{proof}

We are---at last---able to state a precise definition of genericity, or ``completely general position,'' of $Q_1, \hdots, Q_n$.

\begin{defn}
Let $\Phi$ be a real, additive gain graph.  Points $Q_1, Q_2, \hdots, Q_n$ have {\emph {general position with respect to $\Phi$}} if $\cH(\Phi;\bQ)$ satisfies the criteria for genericity of Theorem \ref{T1}.  They have {\emph {completely general position with respect to $\Phi$}} if they have both ideal general position and general position with respect to $\Phi$.
\end{defn}


\section{Matroids and general position} \mylabel{matroids} 

\subsection{Variations on general position} \mylabel{gp}

Re-expressing our results in terms of matroids is technical but illuminating (for matroid theorists) and makes it possible to prove a number of interesting facts with little difficulty.  The essential matroids are $G(\Gamma)$, the graphic (polygon) matroid of a graph $\Gamma$, and $L_0(\Phi)$, the complete lift matroid of the gain graph $\Phi$.  The point set of $L_0(\Phi)$ is $E \cup \{e_0\}$, where $e_0 \notin E$, and its rank function is $\rk(S) = n - c(S \setminus \{e_0\}) + \epsilon$ for $S \subseteq E \cup \{e_0\}$, where $\epsilon = 0$ if $S$ is a balanced subset of $E$ and $1$ otherwise.  For more about balanced flats and $L_0$ we refer to \cite[\S\S II.2, II.3]{BG}.  The geometric lattice of flats of a matroid $M$ is denoted by $\Lat M$ and the geometric semilattice of balanced flats of $L_0(\Phi)$ is $\Latb\Phi$.  Truncation of a matroid or of its lattice of flats to rank $r$ is denoted by $T_r$; this means the elements of ranks $r$ and higher, if any, are replaced by the single element $\hat1$.

We begin with six definitions of general position of $\bQ$.  The first four restate those already given in Sections \ref{intro} and \ref{pyth}.  Recall that $\Pinf$ is the set of all ideal points $p_{ij}$ on the reference lines $Q_iQ_j$.  Its projective dependence matroid is $M(\Pinf)$.
\begin{itemize}
\item \mylabel{SP}
\emph{Simple position (SP)}:
\item[] The matroid of the reference points is the uniform matroid of rank $\min(n,d+1)$.
\item \mylabel{IGP}
\emph{Ideal general position (IGP)}:  
\item[] $M(\Pinf) \cong T_d(G(K_n))$, the matroid of the complete graph truncated to rank $d$, under the canonical correspondence $p_{ij} \leftrightarrow ij$.  (IGP entails that all reference points are distinct, since by our definition that $Q_iQ_j = \aff(Q_i,Q_j)$, $p_{ij} = (Q_iQ_j)\atinf = \eset$ if $Q_i = Q_j$.  Then $p_{ij}$ is, in effect, a matroid loop in $M(\Pinf)$, which consequently is not isomorphic to $T_d(G(K_n))$, not even when $d=1$.)
\item \mylabel{GP}
\emph{General position with respect to $\Phi$ (GP)}:  
\item[] $\cL(\cH(\Phi;\bQ)) \cong [\Latb\Phi]_0^d$, the geometric semilattice consisting of ranks $0$ through $d$ of $\Latb\Phi$, under the canonical correspondence $h(e) \leftrightarrow e$.
\item \mylabel{CGP}
\emph{Completely general position with respect to $\Phi$ (CGP)}:  
\item[] The combination of IGP and GP.
\item \mylabel{PGP}
\emph{Projective general position with respect to $\Phi$ (PGP)}:  
\item[] $\cL(\cH_\bbP(\Phi;\bQ)) \cong T_{d+1}(\Lat L_0(\Phi))$ under the canonical correspondence $h(e) \leftrightarrow e$ and $h_\infty \leftrightarrow e_0$.
\item \mylabel{PCGP}
\emph{Projective completely general position with respect to $\Phi$ (PCGP)}:  
\item[] The combination of IGP and PGP.
\end{itemize}

The meaning of Theorem \ref{T1} is that, generically, $\bQ$ has GP.  We can say more.

\begin{thm} \mylabel{C1} 
Let $\Phi$ be a fixed real, additive gain graph on $n$ vertices.  Assume $n > d \geq 1$.  Generically, $n$ reference points $Q_1,\hdots,Q_n$ in $\bbE^d$ have completely general position and projective completely general position with respect to $\Phi$.
\end{thm}

\begin{proof} 
GP is a translation of Theorem \ref{T1} into the language of gain-graphic matroids.  PGP is equivalent to GP by Theorem \ref{Tgp}(ii).  IGP is a generic property of $n$ points by \cite[\S2.1]{Mason}, as we explain in the proof of Proposition \ref{Ppar}.
\end{proof}

\begin{prop} \mylabel{Ppoint}
If $\bQ$ has ideal general position and, for every $\cS \subseteq \cH$ that corresponds to a forest of $d+1$ edges in $\Phi$, $\bigcap \cS$ is empty, then $\bQ$ has general position with respect to $\Phi$.
\end{prop}

\begin{proof} 
This follows from Lemma \ref{L5} and its proof.
\end{proof}

We should like to say something about the relationships among these several notions of general position.  We saw in Example \ref{Xnongeneric} that IGP does not imply GP.  This suggests that both are needed in the definition of completely general position.  Still, IGP is redundant if $\Phi$ is \emph{complete}, that is, every pair of vertices is adjacent.  We call $\Phi$ \emph{separable} if it has a vertex $v$ such that $\Phi \setminus v$ is disconnected.

\begin{thm} \mylabel{Tgp}
Given: a real, additive gain graph $\Phi$ on $n$ vertices and points $Q_1$, \ldots, $Q_n \in \bbE^d$.
\equivslistbegin
\item\mylabel{Tgp:sp} Ideal general position implies simple position.
\item\mylabel{Tgp:gp} General position and projective general position with respect to $\Phi$ are equivalent.
\item\mylabel{Tgp:cgp} Completely general position and projective completely general position with respect to $\Phi$ are equivalent.
\item\mylabel{Tgp:igp} General position with respect to $\Phi$ implies ideal general position if $\Phi$ is complete, but not if $\Phi$ is disconnected or separable.
\equivslistend
\end{thm}

\begin{proof} 
We prove (\ref{Tgp:sp}) by contradiction.  Suppose a set $D$ of $Q_i$'s is minimally affinely dependent and has $\delta = |D| \leq d+1$.  Then $\dim D = \delta - 2 \leq d-1$.  IGP implies that $\Pinf(D) = \{ p_{ij} : Q_i, Q_j \in D \}$ has matroid $T_d(G(K_\delta))$, so $\dim \Pinf(D) = \min(d-1,\delta-2) = \delta - 2$.  But $\Pinf(D) \subseteq (\aff D)\atinf$, whose dimension $= \dim D - 1 = \delta - 3$.  Thus no such $D$ can exist.

For the rest we need a basic fact from \cite{WW}.  $[x,y]$ denotes the interval from $x$ to $y$.

\begin{lem}[Wachs and Walker {\cite[Thm.\ 3.2]{WW}}] \mylabel{Lww}
Up to isomorphism, a geometric semilattice has the form $L \setminus [e_0,\hat1]$ for one and only one geometric lattice $L$ and atom $e_0$ of $L$.
\end{lem}

In other words, a geometric semilattice determines the whole geometric lattice of which it is a part.  This is very important.

Here is how we use Lemma \ref{Lww}.  Set $\cH = \cH(\Phi;Q_1,\hdots,Q_n)$.  First we note that $[\Latb \Phi]_0^d = T_{d+1}(\Lat L_0(\Phi)) \setminus [e_0,\hat1]$, so $T_{d+1}(\Lat L_0(\Phi))$ is determined by $[\Latb \Phi]_0^d$.  Second, $\cL(\cH)$ determines $\cL(\cH_\bbP)$ since it equals $\cL(\cH_\bbP) \setminus [h_\infty,\hat1]$.  Combining these, GP $\implies$ PGP.  The converse is trivial.  Thus we have proved \eqref{Tgp:gp} and \eqref{Tgp:cgp}.

For \eqref{Tgp:igp} the essential fact is that $(\Lat L_0(\Phi))/e_0$ (the interval above $e_0$) equals $\Lat G(\Gamma)$ if $\Gamma$ is the underlying graph of $\Phi$.  Since $\cL(\cH_\bbP^{h_\infty})$, where $\cH_\bbP^{h_\infty} = \{ h_\bbP \cap h_\infty : h \in \cH \}$, is the interval $[h_\infty,\hat1] \subseteq \cL(\cH_\bbP)$, we know that $\cL(\cH_\bbP^{h_\infty}) \cong T_d(\Lat G(\Gamma))$ under the canonical correspondence $ h(e)\atinf \longleftrightarrow$ the atom of $\Lat G(\Gamma)$ that contains $e$.  If $\Phi$ is complete, $\Lat G(\Gamma) = \Lat G(K_n)$, so we are done.

The assertions about disconnected and separable $\Phi$ follow from Example \ref{Xnsp}.
\end{proof}

\begin{exam} \mylabel{Xnsp}
Let $\Phi$ be a gain graph that is the union of $\Phi_1$ and $\Phi_2$ where $|V_1 \cap V_2| = \delta \leq 1$, $n_1 = |V_1| \geq 2$, and $n_2 = |V_2| \geq 2$.  Then $n = |V| = n_1 + n_2 - \delta$.  Choose $d_1, d_2 > 0$ so that $d = d_1 + d_2$ satisfies $n \geq d + 3 - \delta$.  Choose reference points $Q_1^{(k)}, \hdots, Q_{n_k}^{(k)} \in \bbE^{d_k}$ for $\cH^{(k)} = \cH(\Phi_k; Q_1^{(k)}, \hdots, Q_{n_k}^{(k)})$ that have general position with respect to $\cH_k$.  Now embed $\bbE_{d_1}$ into $\bbE^d$ as $\bbE_{d_1} \times \{0\}$ and embed $\bbE_{d_2}$ into $\bbE^d$ as $\{0\} \times \bbE_{d_2}$.  If $\delta = 1$, choose coordinates in each $\bbE^{d_k}$ so that $Q_1^{(k)}$ is at the origin; then the embedding identifies $Q_1^{(1)}$ and $Q_1^{(2)}$.  If $\delta = 0$, choose coordinates so no $Q_i^{(k)}$ is at the origin.

There is an affinely dependent subset of $Q_1^{(k)}, \hdots, Q_{n_k}^{(k)}$ if $n_k \geq d_k + 2$.  The only way to avoid this for both $k = 1$ and $k = 2$ is to have $n_1+n_2 \leq d+2$, contrary to our choice of $d$.  Thus the set of all $n$ reference points has nonsimple position in $\bbE^d$.

This analysis overlooks the cases in which $n_1$ or $n_2 = 1$, but they are trivial.
\end{exam}

\begin{exam} \mylabel{Xnigp}
It would be nice to have also an example in which the reference points have simple position as well as general position with respect to $\Phi$ but not ideal general position.  Such examples exist at least for nontrivially disconnected gain graphs.  We outline one such example.

The idea is to embed $\bbE^{d_1}$ and  $\bbE^{d_2}$ in $\bbE^d$ as in Example \ref{Xnsp} and then to find a line and parallel hyperplane $h_Q$ that are affinely spanned by the reference points.  To assure simple position we must take $d_k = n_k-1 > 0$ so $d = n-2$.  The line is $Q_1^{(1)}Q_1^{(2)}$; the remaining points span the hyperplane $h_Q$.  In general we expect $Q_1^{(1)}Q_1^{(2)}$ and $h_Q$ to meet in $\bbE^d$.  However, if we apply a dilatation around $Q_1^{(2)}$ to $\bbE^{d_2}$ with expansion factor $\lambda > 0$, then generically there will be a choice of $\lambda$ that makes $Q_1^{(1)}Q_1^{(2)} \parallel h_Q$.  That precludes IGP, as Proposition \ref{Ppar} will show.
\end{exam}

\begin{exam} \mylabel{Xd2}
Let $d=2$ and $n \geq 4$ and take $\Phi = (K_n \setminus e_{12}, 0)$, that is, the complete graph less one edge, all edges having gain $0$.  The corresponding hyperplane arrangement consists of all but one of the perpendicular bisectors between $n$ points in the plane.

Choose the reference points so that $Q_1Q_2 \parallel Q_3Q_4$ but there are no other parallelisms.  The reference points do not have IGP because $p_{12} = p_{34}$, so $M(\Pinf) \neq T_2(M(K_n))$.  Since the arrangement of bisectors can clearly be made to have general position with respect to $\Phi$ by deforming the reference points slightly without losing the one required parallelism, we have an example in which $\Phi$ is as nearly complete as possible, yet which violates the conclusion of Theorem \ref{Tgp}\eqref{Tgp:igp}.

However, I have not seen how to make such an example in dimension $d > 2$.  Thus it is still possible that the hypothesis of completeness can be weakened in higher dimensions.
\end{exam}

These examples give some idea of the interrelations between simple and ideal general position on the one hand and general position with respect to $\Phi$ on the other, but they leave us far from a complete understanding.  Here are some obvious questions.

\begin{resprob}  \mylabel{RPgp}
(a) Is it possible to have GP and SP but not IGP with $n > d+2$?  If $\Phi$ is connected?

(b) Show that it is possible to have GP without IGP for inseparable gain graphs, in particular those with $2$-separations (and enough vertices).
\end{resprob}

\begin{prop} \mylabel{Pfree}
If the reference points are affinely independent, they have completely general position with respect to any gain graph of order $n$.
\end{prop}

\begin{proof}
It is clear that they have ideal general position, indeed $M(\Pinf) = G(K_n)$.  Since $n \leq d+1$, $\Phi$ cannot contain a forest of $d+1$ edges.  So Proposition \ref{Ppoint} is vacuously satisfied.
\end{proof}


\subsection{At infinity} \mylabel{infin}

It may not be clear from our definition that ideal general position has any connection with parallelism.  Here we settle that.  A \emph{parallelism} within reference points $Q_1,\hdots,Q_n \in \bbE^d$ means a pair of flats, $f$ and $g$, generated by reference points so that neither is a point and $\eset \neq f_\bbP \cap g_\bbP \subseteq h_\infty$.  Equivalently, $d > \dim f, \dim g > 0$ and $\dim f + \dim g \geq \dim \aff(f \cup g)$, but $f \cap g = \eset$.

Again we employ the convenient, short notation $s\atinf = s_\bbP \cap h_\infty$ for an affine flat $s$.

\begin{prop} \mylabel{Ppar}
Points $Q_1,\hdots,Q_n \in \bbE^d$ have IGP if and only if they have no parallelisms.
\end{prop}

\begin{proof}
Call a projective hyperplane \emph{singular} if it contains either a reference point or the intersection of two flats that are generated by reference points and have nonempty intersection.  We can interpret the existence of a parallelism as resulting from having the reference points fixed in advance and choosing the ideal hyperplane so that it is singular although it contains no reference point.

A theorem of Mason (\cite[\S 2.1]{Mason}, or see \cite[Prop.\ 7.7.8]{Bryl}) says that, if $n$ points in $\bbP^d$ have matroid $M$, then a nonsingular hyperplane intersects the lines generated by those points in a point set whose matroid is $D_1(M)$, the first Dilworth lower truncation of $M$.  In our case $M = U_{d+1}(n)$, the uniform matroid, so $D_1(M) = T_{d+1}(G(K_n))$.  Thus if there are no parallelisms within the reference points, then $M(\Pinf) = T_d(G(K_n))$.  That proves IGP.

Suppose on the other hand that the reference points have IGP (hence simple position) and yet $h_\infty$ is singular without containing a reference point.  Say $h_\infty \supseteq f_\bbP \cap t_\bbP$, where $f$ and $g$ are spanned by $k$ and $l$ reference points, respectively, with $2 \leq k, l \leq d$.  Then $\dim f = k-1$, $\dim g = l-1$, $\dim (f \cup g) = \min(d, k+l-1)$, and $\dim(f_\bbP \cap g_\bbP) = k + l - 2 - \dim(f \cup g) = \max(k+l-d-2,-1)$.  Since $f_\bbP \cap g_\bbP \neq \eset$, $k+l-d-2 \geq 0$; that is, $k+l \geq d+2$.  Now, there are $\binom{k}{2}$ lines determined by the generators of $f$ and $\binom{l}{2}$ determined by the generators of $g$, which meet $h_\infty$ in a set $\Pinf'$ of points whose matroid, by IGP, is $T_d(G(K_k \cup K_l))$, where the union is disjoint.  This matroid has rank $= \min(k+l-2, d) = d$, whence $\dim \Pinf' = d-1$.  Yet $\Pinf'$ lies in $(f\atinf) \cup (g\atinf)$, whose dimension is $\dim(f_\bbP \cap h_\infty) + \dim(g_\bbP \cap h_\infty) - \dim(f_\bbP \cap g_\bbP \cap h_\infty) = \dim f + \dim g - 2 - \dim(f_\bbP \cap g_\bbP) = \dim(f \cup g) - 2 \leq d-2$.  Thus IGP contradicts singularity of $h_\infty$; we deduce the backward implication of the proposition.
\end{proof}


\subsection{Ideal points} \mylabel{pij}

To study dissections within a flat (Section \ref{induced}) we need more information about the set $\Pinf$ and especially about its subsets
$$
\Pinf(s) = \big\{ p_{ij} : \exists e \text{ for which } V(e) = \{i,j\} \text{ and } h(e) \supseteq s \big\}
$$
for $s \in \cL(\cH(\Phi;\bQ))$ and
$$
\Pinf(S) = \{ p_{ij} : [i] = [j] \text{ in } \pi(S) \},
$$
where $\pi(S)$ is the partition of $V$ induced by the edge set $S$ in $\Phi$.  Recalling that $E(s) = \{ e \in E(\Phi) : h(e) \supseteq s \}$, it is clear that
$$
\Pinf(s) \subseteq \Pinf(E(s)).
$$

\begin{lem} \mylabel{Lpij}
For a point $p_{ij} \in \Pinf$ and a flat $s \in \cL(\cH(\Phi;\bQ))$, the following properties are equivalent:
\equivslistbegin
\item\mylabel{Lpij:pij} $p_{ij} \in s\atinf^*$.
\item\mylabel{Lpij:h} There is a hyperplane $h \perp Q_iQ_j$ such that $h \supseteq s$.
\item\mylabel{Lpij:l} $Q_iQ_j \perp s$ (where we assume $s$ is not a point).
\equivslistend
\end{lem}

\begin{proof}
The equivalence of \eqref{Lpij:h} and \eqref{Lpij:l} is obvious.

As for that of \eqref{Lpij:pij} and \eqref{Lpij:h}, first note that $p_{ij} \in s\atinf^* \iff p_{ij}^* \supseteq s\atinf$.  Now, in one direction, if $h \perp Q_iQ_j$ and $h \supseteq s$, then $p_{ij}^* = h\atinf \supseteq s\atinf$.  Conversely, if $p_{ij}^* \supseteq s\atinf$, let $P \in s$, $h_\bbP = \Span(p_{ij}^* \cup \{P\})$, and $h =$ the finite part of $h_\bbP$.  Then $h$ satisfies the conditions of \eqref{Lpij:h}, because $h\atinf = p_{ij}^* \iff h \perp Q_iQ_j$.
\end{proof}

We write $\proj_s P$ for the orthogonal projection of a point $P \in \bbE^d$ into a flat $s$.

\begin{lem} \mylabel{Lproj}
If $[i] = [j]$ in $\pi(E(s))$, then $\proj_s Q_i = \proj_s Q_j$.
\end{lem}

\begin{proof}
Let us first look at the case in which $i$ and $j$ are adjacent by an edge $e \in E(s)$.  Then $h(e) \supseteq s$, hence $Q_iQ_j \perp s$ by Lemma \ref{Lpij}, whence $Q_i$ and $Q_j$ have the same projection.

In general, $[i] = [j]$ means that $i$ and $j$ are joined by a path $i = i_0, e_1, i_1, \hdots, e_l, i_l = j$ in $E(s)$.  Thus $\proj_s Q_{i_0} = \proj_s Q_{i_1} = \dots = \proj_s Q_{i_l}$.
\end{proof}

\begin{lem} \mylabel{Lptn}
$\Pinf(E(s)) \subseteq s\atinf^*$.
\end{lem}

\begin{proof}
$p_{ij} \in \Pinf(E(s)) \iff$ (by definition) $[i] = [j] \implies$ (by Lemma \ref{Lproj}) $\proj_s Q_i = \proj_s Q_j$ $\iff Q_iQ_j \perp s \iff$ (by Lemma \ref{Lpij}) $p_{ij} \in s\atinf^*$.
\end{proof}

\begin{lem} \mylabel{Lsinf}
If $\bQ$ has ideal general position, then $\Span \Pinf(s) = s\atinf^*$.
\end{lem}

\begin{proof}
First, $p_{ij} \in \Pinf(s) \iff \exists h(e) \supseteq s$ with $V(e) = \{i,j\} \implies p_{ij} \in s\atinf^*$ by Lemma \ref{Lpij}.  Thus $\Pinf(s) \subseteq s\atinf^*$.

Now we compare dimensions.  Because $M(\Pinf) \cong T_d(G(K_n))$ (by ideal general position) and $E(s)$ is balanced (by Lemma \ref{L3}), $\dim \Pinf(s) = \min(d-1, \rk E(s)-1)$.  If $s$ is not a point, $\rk E(s) = \codim s \leq d$ (by Lemma \ref{L5}) so $\dim \Pinf(s) = \rk E(s) - 1$.  At the same time, $\dim s\atinf^* = d-2 - \dim s\atinf = d-1 - \dim s = \rk E(s) - 1$.  Since this equals $\dim \Pinf(s)$ while also $\Pinf(s) \subseteq s\atinf^*$, $\Pinf(s)$ spans $s\atinf^*$ if $s$ is not a point.  If $s$ is a point, then $\rk E(s) \geq \codim s = d$ (again by Lemma \ref{L5}), so $\dim \Pinf(s) = d-1$.  Consequently, $\Pinf(s)$ spans $h_\infty = s\atinf^*$.
\end{proof}

\begin{prop} \mylabel{Pptn}
Suppose $\bQ$ has ideal general position and $s \in \cL(\cH(\Phi;\bQ))$ but $s$ is not a point.  Then $\Pinf(E(s)) = \Pinf \cap \Span \Pinf(s)$, the closure of $\Pinf(s)$ in $M(\Pinf)$.
\end{prop}

\begin{proof}
We know that $\Span \Pinf(s) = s\atinf^* \supseteq \Pinf(E(s)) \supset \Pinf(s)$.  Thus $\Pinf(E(s))$ spans $s\atinf^*$. However, it may not equal $\Pinf \cap \Span \Pinf(s)$ if $s$ is a point, so we now assume $\dim s > 0$.  Therefore $\rk E(s)) = \codim s \leq d-1$, from which we see that $\dim s\atinf^* > d-1$ so $\rk \Pinf(E(s)) < d$ in $M(\Pinf)$.  Since $M(\Pinf) \cong T_d(G(K_n))$ and $\Pinf(E(s))$ (from its definition) corresponds to a flat in $G(K_n)$ of rank $< d$, $\Pinf (E(s))$ is itself a flat in $M(\Pinf)$.  That means $\Pinf(E(s)) = \Pinf \cap \Span \Pinf(E(s)) = \Pinf \cap s\atinf^*$.  Because $\Pinf(s)$ also spans $s\atinf^*$, we conclude that $\Pinf(E(s)) = \Pinf \cap \Span \Pinf(s)$, as desired.
\end{proof}

\begin{cor} \mylabel{Pproj}
Suppose $\bQ$ has ideal general position and $s \in \cL(\cH(\Phi;\bQ))$ is not a point.  Then $\proj_s Q_i = \proj_s Q_j \iff [i] = [j]$ in $\pi(E(s))$.
\end{cor}

\begin{proof}
Sufficiency is Lemma \ref{Lproj}.  For necessity, suppose that $Q_i$ and $Q_j$ have the same porjection.  Then $Q_iQ_j \perp s$, so by Lemma \ref{Lpij}, $\_{ij} \in s\atinf^*$.  But then by Proposition \ref{Pptn}, $p_{ij} \in \Pinf(E(s))$.  That is, $[i] = [j]$.
\end{proof}


\section{Cross-sections} \mylabel{xsect} 

Peering into a flat $t$ in $\bbE^d$, we see it dissected into pieces by the hyperplanes of $\cH(\Phi;$ $Q_1, \hdots, Q_n)$ that do not contain it.  Does this induced dissection result from a Pythagorean arrangement of relative hyperplanes, as Good and Tideman said of certain flats of their original arrangement?  Indeed it does---and this will enable us to explain why a flaw in Good and Tideman's proof does not invalidate their enumerative results.  In order to explain how an induced arrangement is Pythagorean, we need notions of contraction and therefore of switching of a gain graph.

\subsection{Switching and contraction} \mylabel{switch}
A \emph{switching function} on a real, additive gain graph $\Phi$ is simply a function $\eta : V(\Phi) \to \bbR^+$.  The secret is in how to use it.  \emph{Switching $\Phi$ by $\eta$} means replacing the gain function $\phi$ by $\phi^\eta$ whose definition is
$$
\phi^\eta(e;i,j) = \phi(e;i,j) - \eta(i) + \eta(j).
$$
We write $\Phi^\eta$ for the switched gain graph.  (Since $\Phi^\eta$ is unaffected by adding a constant to $\eta$, we can always take $\eta \geq 0$.)  The difference between $\cH(\Phi;\bQ)$ and $\cH(\Phi^\eta;\bQ)$ is in the appearance of offsets in the locations of perpendiculars.  Instead of being placed at 
$\psi_{ij}(h) = \phi(e;i,j)$, the hyperplane $h(e)$ is now located at 
$\psi_{ij}(h) = \phi(e;i,j) - \eta(i) + \eta(j)$.  In full, the relocated 
$h(e)$ consists of the points $P$ for which
$$
[d(P,Q_i)^2 + \eta(i)] - [d(P,Q_j)^2 + \eta(j)] = \phi(e;i,j).
$$
Thus $\eta$ modifies the location rule by offsetting $h(e)$ along the line $Q_iQ_j$.  Switching can change the combinatorial type of $\cL(\cH)$ and therefore also of $\cH$ because it can put the reference points into special position with respect to $\Phi^\eta$ even though they were general with respect to $\Phi$.

(Pythagorean arrangements of the form $\cH((K_n,0)^\eta;\bQ)$ with $\eta \geq 0$ correspond to the weighted Voronoi diagrams known as ``power diagrams''.  There $\eta(i)$, sometimes called the ``weight of $Q_i$'', is interpreted as the squared radius of a sphere $S_i$ centered at $Q_i$; thus $d(P,Q_i)^2 - \eta(i) = t_i(P)^2$, the squared length of a tangent from $P$ to $S_i$, provided $P$ lies outside all spheres.  The hyperplane $h_{ij}$ is located where $t_i(P) = t_j(P)$.\footnote{Nonnegative weights have a long history; see \cite{Auren} or \cite[Section 13.6]{Edels}.  Arbitrary real weights appear to be rare (an exception is \cite{AB}), possibly because they do not have the nice power-diagram interpretation.  The literature, however, concerns only nearest reference points (as weighted, i.e., after allowing for offsets) and variations in the same spirit and consequently uses only pieces of dissecting hyperplanes instead of entire hyperplanes.  It is also limited to what we should call all-zero gains.})

Now we define contraction of a balanced edge set in $\Phi$.  Suppose $S \subseteq E(\Phi)$ has all zero gains; that is, $\phi\big|_S \equiv 0$.  Let $\pi(S)$ be the partition of the vertices implied by the connected components of $(V,S)$ and 
let $[i]$ denote the block of $\pi(S)$ that contains $i$.  The \emph{contraction} 
$\Phi/S$ has vertex set $\pi(S)$ and edge set $E(\Phi) \setminus S$.  An edge 
$e$ with endpoints $i,j$ in $\Phi$ has endpoints $[i], [j]$ in $\Phi/S$.  (This process may introduce loops.  We discard any loops whose gain is nonzero.)  Since the definition permits us to contract only zero-gain edges, we need switching to contract an arbitrary balanced edge set $S$.  
There is always a switching function $\eta$ such that $\phi^\eta\big|_S \equiv 0$.  
We choose such a function, switch $\Phi$ by it, and contract $S$.  Thus we get 
\emph{a contraction} $\Phi^\eta/S$ of $\Phi$.  This contraction is not unique: there is a different one for each choice of $\eta$ (up to an additive constant).  (In that respect our definition is more refined than the standard one found, e.g., in \cite[\S I.5]{BG}.  Usually it is enough for a gain graph to be specified up to 
switching; thus $\Phi$ and $\Phi^\eta$ would be equivalent and it would not matter how we switch $\Phi$ because all contractions $\Phi^\eta/S$ would also be equivalent.  But that is not so here, since switching changes the hyperplane arrangement.  In discarding nonzero loops our definition is cruder than the usual one; we delete them because they correspond to nothing in $\bbE^d$.)  A fact we shall have use for is that, if $\Phi$ is balanced, then $\Phi/S$ is again balanced.

We also need \emph{collapsing by a partition} $\pi$ of the vertex set $[n]$: we take $\Phi/\pi$ to be the gain graph with the same edges and the same gain function as in $\Phi$ but with all the vertices in each block of $\pi$ identified to a point and with all loops removed.  This is much like contraction, but for consistency with general custom we confine the latter name to contraction by an edge set.

\subsection{Dissections within a flat} \mylabel{induced}
To describe the induced dissection of an affine flat $t$ we still need a few short definitions.  The \emph{induced arrangement} in $t$ is 
$$
\cH^t = \big\{ h \cap t : h \in \cH, h \not\supseteq t, \text{ and } h \cap t \neq \eset \big\}.
$$
Write $\proj_t P$ for the orthogonal projection of a point $P$ into $t$ and let $\pi(t)$ be the partition of $[n]$ that corresponds to the equivalence relation defined by $i \sim j$ if $\proj_t Q_i = \proj_t Q_j$.  Set $\eta_t(i) = d(Q_i,t)^2$.

\begin{thm} \mylabel{Tsect}
Take a fixed real, additive gain graph $\Phi$ of order $n$ and points 
$Q_1, \hdots, Q_n \in \bbE^d$.  
Let $\cH = \cH(\Phi; Q_1, \hdots, Q_n)$.  For an affine flat $t$, let 
$\bar Q_1, \hdots, \bar Q_k$ be the distinct points $\proj_t Q_i$; thus 
$k = |\pi(t)|$.  Then
$$
\cH^t = \cH(\Phi^{\eta_t}/\pi(t); \bar Q_1, \hdots, \bar Q_k).
$$

If $t$ is generic of codimension $d'$, then $\cL(\cH^t)$ is isomorphic to the semilattice of all flats of $\cH$ of dimension at least $d'$.

Furthermore, if $Q_1, \hdots, Q_n$ have ideal general position, then $\bar Q_1, \hdots, \bar Q_k$ also have it.
\end{thm}

\begin{lem} \mylabel{Lprojeq}
For a hyperplane $h(e) \in \cH$, corresponding to an edge $e$ with $V(e) = \{i,j \}$, to have $h(e) \supseteq t$ or $h(e) \cap t = \eset$ it is necessary and sufficient that $\proj_t Q_i = \proj_t Q_j$.
\end{lem} 

(This is a version of Corollary \ref{Pproj} suitable for arbitrary affine flats.)

\begin{proof}
Let us compare a point $P \in t$ to the trace in $t$ of $h(e)$, where $V(e) = \{i,j\}$.  We write $Q_i^t = \proj_t Q_i$.  We have
\begin{align*}
d(Q_i,P)^2 &= d(Q_i,Q_i^t)^2 + d(Q_i^t,P)^2, \\
d(Q_j,P)^2 &= d(Q_j,Q_j^t)^2 + d(Q_j^t,P)^2,
\end{align*}
so by subtracting,
$$
\psi_{ij}(P) = \psi^t_{ij}(P) + d(Q_i,Q_i^t)^2 - d(Q_j,Q_j^t)^2,
$$
where $\psi_{ij}^t$ denotes the Pythagorean coordinate along $Q_i^tQ_j^t$.  This expression, together with the fact that $\psi_{ij}(P) = \phi(e;i,j)$, shows that the correct gain for $e$ in the Pythagorean gain graph $\Psi(\cH^t)$ is $\phi^{\eta_t}(e;i,j)$.  

The calculation also proves the lemma, since $h(e) \nsupseteq t$ and $h(e) \cap t \neq \eset$ if and only if $t$ contains points that take different values of the Pythagorean coordinate $\psi_{ij}$, hence of $\psi^t_{ij}$; but such points exist with respect to $\psi^t_{ij}$ if and only if $Q_i^t \neq Q_j^t$.

The coalescence of reference points under projection shows that each block of $\pi(t)$ must be identified to a vertex in $\Psi(\cH^t)$ but there should be no other identifications of reference points.  The lemma shows us which edges of $\Phi$ no longer correspond to hyperplanes in $\cH^t$ so should be deleted.  Thus $\Psi(\cH^t)$ is precisely the gain graph $\Phi^{\eta_t}/\pi(t)$.

The assertion about generic $t$ follows because a flat $s \in \cL(\cH)$ intersects $t$ in the smallest dimension permitted by the modular law: in dimension $\dim s - \codim t$ if that is nonnegative; otherwise the intersection is empty.

Suppose $\bar Q_1, \hdots, \bar Q_k$ fail to have ideal general position.  Then they have a parallelism: flats $\bar f$ and $\bar g$, generated by $\bar Q_i$'s, neither one a point, such that $t\atinf \supseteq \bar f_\bbP \cap \bar g_\bbP \neq \eset$.
Let $t'$ be an orthogonal complement of $t$: a flat such that $t' \perp t$, $t \cap t'$ is a point, and $\dim t + \dim t' = d$.  We can regard $\bbE^d$ as $t \times t'$.  Let $\hat f = \bar f \times t'$ and $\hat g = \bar g \times t'$.  Since $\bar f \cap \bar g = \eset$, $\hat f \cap \hat g = \eset$ as well.

Now let $f = \aff\{Q_i : Q_i \in \hat f\}$ and let $g$ be similar.  Neither $f$ nor $g$ is a point.  They are disjoint because $\hat f \cap \hat g = \eset$, but they meet at infinity because $f_\bbP \cap g_\bbP \supseteq \bar f_\bbP \cap \bar g_\bbP$.  Therefore there is a parallelism within $Q_1, \hdots, Q_n$.

It follows that ideal general position of $Q_1, \hdots, Q_n$ implies ideal general position of $\bar Q_1, \hdots,$ $\bar Q_k$.
\end{proof}

We are especially interested in what happens when we look inside a flat of the arrangement itself.
If $s$ is a flat of $\cH$, we write $\cH(s) = \{ h \in \cH : h \supseteq s \}$ and 
$E(s) = \{ e \in E(\Phi) : h(e) \supseteq s \}$.  $E(s)$ is necessarily balanced, by Lemma \ref{L3}.

\begin{thm} \mylabel{Tinduced}
Take a fixed real, additive gain graph $\Phi$ of order $n$ and points 
$Q_1, \hdots, Q_n \in \bbE^d$ in general position with respect to $\Phi$.  
Let $\cH = \cH(\Phi; Q_1, \hdots, Q_n)$.  For $s \in \cL(\cH)$ which is not a point, let 
$\bar Q_1, \hdots, \bar Q_k$ be the distinct points $\proj_s Q_i$, where 
$k = c(E(s))$.  Then
$$
\cH^s = \cH(\Phi^{\eta_s}/E(s); \bar Q_1, \hdots, \bar Q_k).
$$
and $\bar Q_1, \hdots, \bar Q_k$ have general position in $s$ with respect to $\Phi^{\eta_s}/E(s)$.

Furthermore, if $Q_1, \hdots, Q_n$ have ideal general position, then $\bar Q_1, \hdots, \bar Q_k$ also have it.
\end{thm}

\begin{proof}
Write $S = E(s)$.  We know that $\pi(s) = \pi(S)$, hence $k = c(S)$, due to Corollary \ref{Pproj}.
The bulk of the theorem follows from Theorem \ref{Tsect}.  
What remains to prove is that $\bar Q_1, \hdots, \bar Q_k$ have general position.

By general position of $Q_1, \hdots, Q_n$ with respect to $\Phi$, $\cL(\cH) \cong [\Latb \Phi]_0^d$.  Switching does not change $\Latb \Phi$.  Thus $\cL(\cH) \cong [\Latb \Phi^{\eta_s}]_0^d$.  We also know that $\cL(\cH^s) \cong \cL(\cH)/s = \{ t \in \cL(\cH) : t \geq s \}$; and this is isomorphic to $[\Latb \Phi]_0^d/E(s)$ by the isomorphism of $\cL(\cH)$ with $\Latb \Phi$.  One further fact is needed:  $\Latb(\Phi^{\eta_s}/S) \cong (\Latb \Phi^{\eta_s})/S$ for a balanced edge set $S$ \cite[Prop.\ II.2.7]{BG}.  Tracing through these isomorphisms, we conclude that $\cL(\cH^s) \cong [\Latb(\Phi^{\eta_s}/E(s))]_0^{\dim s}$.  This is the definition of general position of $\bar Q_1, \hdots, \bar Q_k$.
\end{proof}

We can now analyse the oversight in Good and Tideman's Theorem 2 \cite{GT}.  It asserts (in our terminology) that, if $(K_n,0)$ is the complete graph with all-zero gain function $\phi$ and $\cH = \cH((K_n,0);\bQ)$, then generically $\cH^s$ has the form $\cH((K_{n-q+1},0);\bQ^s)$ for some set of reference points $Q_1^s,\hdots,Q_{n-q+1}^s \in s$ if $s \in \cL(\cH)$ is such that $E(s)$ consists of just one nonisolated component that has $q$ vertices (thus $\dim s = d-q+1$).  This kind of flat is what Good and Tideman call a ``circumflat''.  By Theorem \ref{Tinduced} we know that $\cL(\cH^s) \cong [\Latb (K_n,0)^\eta/E(s)]_0^{\dim s}$ where $\eta(i) = d(Q_i,s)^2$; that is, $\cL(\cH^s)$ is an arrangement of perpendiculars with respect to reference points in $s$ and some gain graph.  What we do not know is that reference points exist \emph{within $s$} with respect to which the gain graph can be chosen to have all-zero gains (that is, the hyperplanes are to be perpendicular bisectors); and in fact they may not exist, as Example \ref{Xnbisect} shows.  But $(K_n,0)^\eta/E(s)$ is balanced (in fact, $\cH^s$ is a power-diagram arrangement with centers $bQ$ and spheres of radii $d(Q_i,s)$) and, by Corollary \ref{T3}, generically the face and flat numbers of $\cH(\Phi;\bQ)$ for balanced and complete $\Phi$ are the same as for perpendicular bisectors.  So Good and Tideman's numerical conclusions are correct after all and are valid, moreover, for every flat of $\cH$.

\begin{exam} \mylabel{Xnbisect}
In Figure \ref{Fnbisect} are a planar arrangement of bisectors and a circumflat whose induced arrangement cannot be constructed as bisectors from reference points within the circumflat.
\placefigure{nbisect}
\begin{figure} \mylabel{Fnbisect}
\vbox to 5truein{}
\begin{center}
\includegraphics{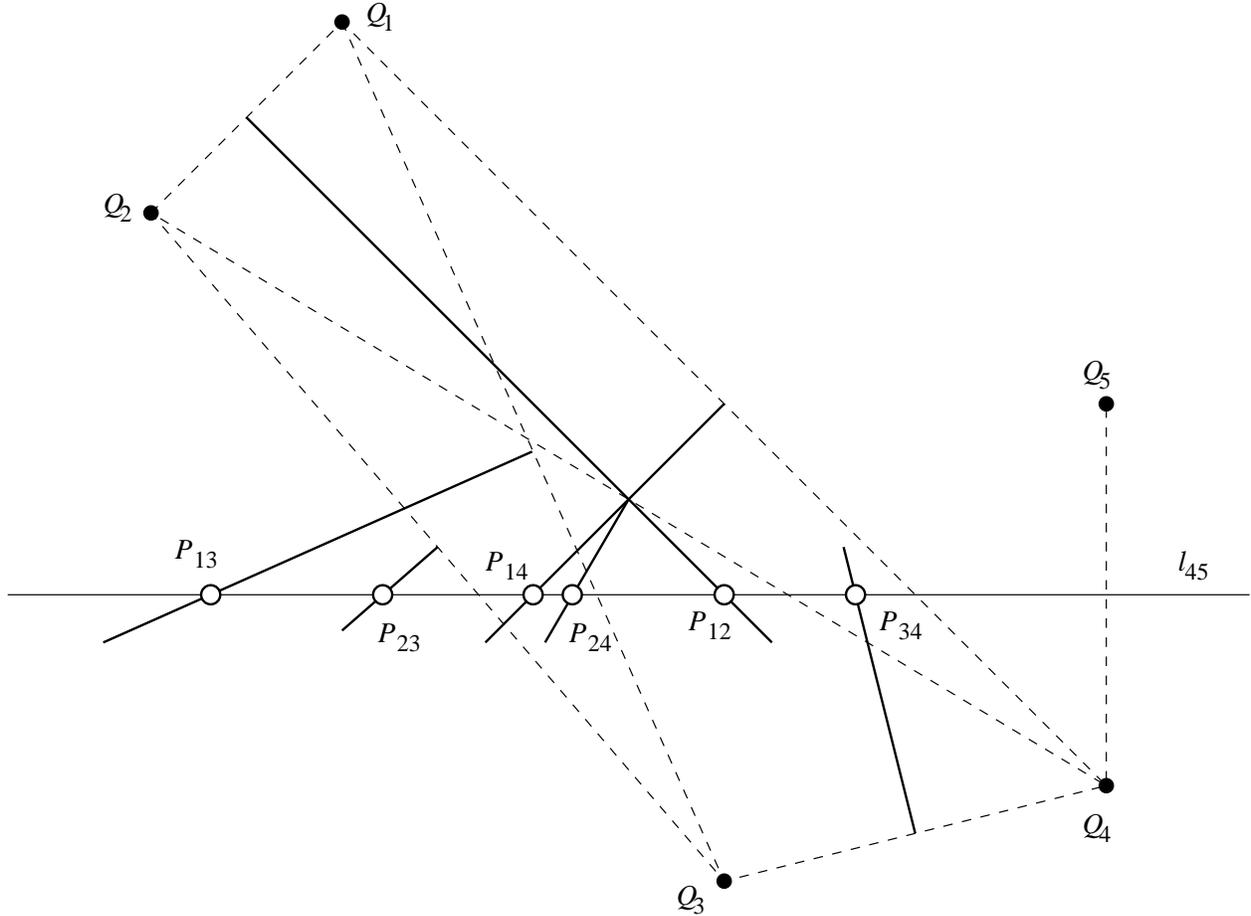}
\end{center}
\caption{A counterexample to the Good--Tideman proof.  The points $P_{ij}$ are not centrally symmetric.}
\end{figure}

First, observe that, if six points in a line are the bisectors of the six segments produced by four reference points in the line, then the bisection points are centrally symmetric.

To construct the example, take generic points $Q_1,\hdots,Q_5$ in the plane.  Let $l_{ij}$ be the perpendicular bisector of $[Q_i,Q_j]$ and for the circumflat take $s = l_{45}$.  The induced arrangement $\cH^s$ consists of one point $P_{ij}$ for each $l_{ij}$ with $i < j \leq 4$; we may ignore $l_{i5}$ since it and $l_{i4}$ meet $s$ in the same point.  Thus we have an arrangement of 6 points in $s$ produced by 6 lines $l_{ij}$ that are the perpendicular bisectors of all the segments generated by $Q_1,\hdots,Q_4$.  In relation to these bisectors, $s$ is an arbitrary line because we can determine it at will by choosing $Q_5$.  It is easy to find $s$ so that the 6 points in $s$ are not centrally symmetric; in fact, that situation is generic. Therefore, $\cH^s$ cannot be produced by reference points in $s$.
\end{exam}

\subsection{Cross-sectional representation} \mylabel{xrep}
Is every Pythagorean arrangement a cross-section of one in which the reference points are an affine basis?  Yes!  

\begin{thm} \mylabel{Txrep}
Given $\Phi$ of order $n$, $d < n-1$, and $Q_1,\dots,Q_n$ that affinely span $\bbE^d$, it is possible to find affinely independent reference points $Q'_1,\dots,Q'_{n} \in \bbE^{n-1}$ such that $\cH(\Phi;\bQ) = \cH(\Phi;\bQ')^t$ for some $d$-flat $t$ in $\bbE^{n-1}$.
\end{thm}

\begin{proof}
Embed $\bbE^{d}$ as a flat $t \subseteq \bbE^{n-1}$ and choose $Q'_i$ that are all at some fixed positive distance $c$ from the corresponding points $Q_i$ but in generically different directions orthogonal to $t$.  Obviously the $Q'_i$ will be affinely independent.  (If $d=n-2$, there are only two directions orthogonal to $t$ so special instructions are necessary.  The reference points have only one minimal dependent set, say $Q_ 1,\ldots,Q_k$.  If $Q_1,\ldots,Q_k$ are not all shifted in the same direction, then $Q'_1,\ldots,Q'_n$ will be affinely independent.) Now $Q_i = \proj_t Q'_i$ and therefore $\cH(\Phi;\bQ')^t = \cH(\Phi^\eta;\bQ)$, where $\eta(i) = d(Q'_i,Q_i) = c$, by Theorem \ref{Tsect}.  Thus $\Phi^\eta = \Phi$.
\end{proof}

Conversely, let us begin with affinely independent reference points in dimension $n-1$.  Then $\cL(\cH(\Phi;\bQ)) \cong \Latb \Phi$ by Theorem \ref{T1} and Proposition \ref{Pfree}.  If $t$ is a generic affine flat of dimension $d$, then $\cL(\cH(\Phi;\bQ)^t) \cong [\Latb \Phi]_0^d$ by Theorem \ref{Tsect}.  This viewpoint gives another explanation of the generic fact that $\cH(\Phi;\bQ)$ in lower dimensions has intersection semilattice $[\Latb \Phi]_0^d$, as Good and Tideman observed for the arrangement of all perpendicular bisectors.  Possibly it could be the basis for an alternative proof of Theorem \ref{T1}.


\section{Non-Pythagorean perpendiculars} \mylabel{nonpyth}

Now let us consider modified Pythagorean description rules, $(\alpha, \Phi)$
where $\alpha \neq 0$.  Recall that the arrangement $\cH(\alpha, \Phi; \bQ)$,
described by $(\alpha, \Phi)$ and based on $\bQ$, has for each edge $e$ (with
endpoints $i$ and $j$) a hyperplane $h(e)$ perpendicular to $Q_iQ_j$ with
Pythagorean coordinate
\begin{equation*}
\psi_{ij}(h(e)) = \phi(e;i,j) d(Q_i,Q_j)^\alpha .
\end{equation*}
Let $Z$ be the set of edges with gain $\phi(e) = 0$.

\begin{thm} \mylabel{T2}
Let $n > d \geq 1$.  Choose a real number $\alpha \neq 0$ and a real, additive
gain graph $\Phi$ on vertices $\{ 1, 2, \hdots, n\}$.  Supposing $\bQ = (Q_1, \hdots, Q_n)$ is generic, then the intersection of a subset $\cS \subseteq \cH(\alpha, \Phi; \bQ)$ is void except that, when $\cS$ corresponds to an edge set $S$ having $n-i$ connected components with $i \leq d$ and in which every circle is contained in $Z$ (that is, the parts of $\cS$ corresponding to circles in $\Phi$ consist of bisectors), then the intersection is a flat of dimension $d-i$.  

Two subsets $\cS_1$ and $\cS_2$ with nonempty intersections have the same intersection if and only if $S_1 \setminus Z = S_2 \setminus Z$ and the connected components of $S_1 \cap Z$ and $S_2 \cap Z$ partition the vertices in the same way.

\end{thm}

\begin{proof}
Let $\Psi$ be the Pythagorean gain graph of $\cH(\alpha, \Phi; \bQ)$; that is,
\begin{equation}  \mylabel{Epsi}
\psi(e;i,j) = \phi(e;i,j) d(Q_i, Q_j)^\alpha .
\end{equation}
Choosing any $\bQ$, by varying $\bQ$ slightly we can ensure that $\bQ$ has ideal general position and the gains in $\Psi$ do not sum to zero on any circle except one whose edges all have gain zero.  (A more abstract way to obtain the same effect is to replace the nonzero gains in $\Phi$ by new real gain values that are linearly independent over the rationals.)  We may now apply Theorem \ref{T1} to $\cH(\Psi;\bQ)$.  This gives the first half of the theorem immediately.

For the second half we see from Theorem \ref{T1} that $S_1 \cup S_2$ must be balanced in $\Psi$ and, if the components of $S_1$ have vertex sets $V_1, \hdots, V_k$, then so do those of $S_2$.  We examine one $V_i$.  Let $S_{1i}$ and $S_{2i}$ be the parts of $S_1$ and $S_2$ with endpoints in $V_i$.

Suppose $S_{1i} \cap Z$ partitions $V_i$ into $B_1, \hdots, B_l$ with $l>1$.  If $S_{2i} \cap Z$ contains an edge $e$ that joins two of the $B_j$'s, say $B_1$ and $B_2$, then $S_{1i} \cup \{e\}$ contains a circle $C \ni e$.  $C \nsubseteq Z$ since $e$ lies in no circle in $(S_{1i} \cap Z) \cup \{e\}$.  Therefore $S_1 \cup S_2$ contains a circle $C$ that is unbalanced in $\Psi$.  This is impossible.  Hence $S_{2i} \cap Z$ does not join any of the $B_j$'s.  It follows that $S_{1i} \cap Z$ and $S_{2i} \cap Z$ partition $V_i$ in the same way; as this is true for every $i$, $S_1 \cap Z$ and $S_2 \cap Z$ partition $V$ in the same way.

Meanwhile, if $S_{2i} \setminus Z$ contains an edge $f \notin S_{1i} \setminus Z$, then $S_{1i} \cup \{f\}$ contains a circle $C \ni f$.  But $C$ is then unbalanced, which contradicts balance of $S_1 \cup S_2$.  Therefore $S_{2i} \setminus Z \subseteq S_{1i}$.  It follows that $S_{1i} \setminus Z = S_{2i} \setminus Z$.
\end{proof}

Comparing with the definition of the complete lift matroid $L_0(\Gamma,\cB)$ of a biased graph (\cite[\S II.3]{BG}; see the gain-graphic version in our Section \ref{gp}) and its balanced flats, we have a matroid-theoretic restatement of Theorem \ref{T2}.

\begin{cor} \mylabel{C2}
Let $\Phi$ be a real, additive gain graph on $n$ vertices, $\alpha$ a nonzero real number, and $n > d \geq 1$.  Let $\Gamma$ be the underlying graph of $\Phi$, let $Z = \{ e \in E: \phi(e) = 0 \}$, and let $\cZ$ be the set of circles in $Z$.  Then for generic $\bQ = (Q_1, \hdots, Q_n)$ in $\bbE^d$, $\cL(\cH(\alpha,\Phi;\bQ)) \cong \big[\Latb(\Gamma,\cZ)\big]_0^d$ and $\cL(\cH_\bbP(\alpha,\Phi;\bQ)) \cong T_{d+1}(\Lat L_0(\Gamma,\cZ))$, with flats of $\cH_\bbP$ in $h_\infty$ corresponding to matroid flats containing the extra point $e_0$. \hfill $\square$
\end{cor}

\begin{exam} \mylabel{Xnpgeneric}
In Example \ref{Xsdgeneric}, $\cZ$ consists of the one circle $e_{12}(0)e_{23}(0)e_{13}(0)$.
\end{exam}\smallskip

Corollary \ref{C2} leads to the question of describing the edge sets $S \in \Latb \Psi$.  From the proof of Theorem \ref{T2}, $S$ is balanced $\iff$ every circle in it is contained in $Z$.

\begin{prop} \mylabel{Pbal}
Let $S \subseteq E(\Psi)$.  $S \in \Latb \Psi$ $\iff$ $S$ is balanced and $e \in S$ whenever $e$ is an edge in $Z$ whose endpoints are connected by $S \cap Z$ $\iff$ $S$ is balanced and $S \cap Z$ is closed in the graphic matroid of $(V,Z)$.
\end{prop}

\begin{proof}
Just interpret the definitions of balance and balance-closure in \cite[\S\S I.5, I.3, resp.]{BG}.
\end{proof}

We should like a similar description of the flats of the induced arrangement $\cH(\alpha,\Phi;\bQ)^s$.  We get it by applying Theorem \ref{Tinduced} to $\Psi$.  The Pythagorean gain graph of $\cH^s$ is $\Psi^s = \Psi^\eta/E(s)$ where $\eta(i) = -d(Q_i,s)^2$.  We wish to determine the edges sets $A$ of $\Psi^s$ that belong to $\Latb \Psi^s$, hence correspond to flats of a generic $\cH^s$ (provided that $A$ has the right rank, for which it is necessary and sufficient that $c_{\Psi^s}(A) \geq n-d$).  To state a relatively nice result, we have to think of $A$ in two ways at once: as an edge set in $\Psi^s$ and also as one in $\Psi$ (which it is, since $E(\Psi^s) \subseteq E(\Psi)$).  An \emph{$S$-subcomponent} is a component of $(V, Z \cap S)$.

\begin{prop} \mylabel{Pindbal}
Let $A \subseteq E(\Psi^s)$.  $A \in \Latb \Psi^s$ $\iff$ $A$ satisfies the two conditions

(a) every circle in $A$ (as an edge set in $\Psi^s$) lies in $Z \cap A$, and

(b) for each component $A_1$ of $A$ (in $\Psi^s$), $A_1$ touches (in $\Psi$) at most one $E(s)$-subcomponent in each component of $(V,E(s))$ and, if $A_1$ touches two $E(s)$-subcomponents that are joined by an edge $e \in Z \setminus E(s)$, then $e \in A$.
\end{prop}

The point of Proposition \ref{Pindbal} is to show that $\Latb\Psi^s$ can be described but not so easily as $\Latb\Psi$.  The proof, which is not hard, is very technical and ungeometrical and really belongs to the theory of fat forests \cite[Ch.\ IV]{GLSP}.  Thus we omit it.


\section{A projectionist's view.} \mylabel{proj}

A quite different geometrical construction for the Pythagorean arrangement of a balanced gain graph was suggested to me by Herbert Edelsbrunner in 1984, based on a parabolic approach to Voronoi diagrams 
that goes back to the original paper \cite{Vor} 
(see Edelsbrunner and Seidel \cite[end of Note 3.1]{ES} for the plane, or the general treatment in \cite[Section 4.1]{Auren} or \cite[Section 13.1, p.\ 296]{Edels}).

Embed $\bbE^d$ with coordinate vectors $x = (x_1, \hdots, x_d)$ in $\bbE^{d+1}$ whose $(d+1)$-st coordinate is $z$.  The {\emph {fundamental paraboloid}} is the hypersurface 
$S$: $z=x_1^2+ \cdots + x_d^2$.  For $P\in \bbE^d$, let $P^*$ be its vertical projection 
up into $S$; for $Q^*\in S$, let $Q$ be its projection down to $\bbE^d$.  Let $T_{Q^*}$ be the tangent $d$-space to $S$ at $Q^* \in S$.  If $Q_i, Q_j \in \bbE^d$, then $h_{ij}$, the vertical projection into $\bbE^d$ of the intersection $T_{Q_i^*} \cap T_{Q^*_j}$, is the perpendicular bisector of the line segment between $Q_i$ and $Q_j$.  Thus the Good--Tideman arrangement can be constructed by 
\begin{enumerate}
\item[(1)]  lifting each $Q_i$ to $Q_i^* \in S$,
\item[(2)]  forming the $\binom{n}{2}$ tangent space intersections, and
\item[(3)]  projecting them back to $\bbE^d$.  
\end{enumerate}
Edelsbrunner suggested a similar procedure with step (2) replaced by 
\begin{enumerate}
\item[(2$'$)]  raise each tangent space $T_{Q_i^*}$ by an arbitrary amount 
$\eta(i) \in \bbR$ to a `displaced tangent' $T_i$, parallel to $T_{Q_i^*}$ but at height $z$ increased by $\eta(i)$, and form the intersections of the displaced tangents.
\end{enumerate}
\begin{lem}\mylabel{L1} 
This procedure forms a Pythagorean arrangement in $\bbE^d$, based on $Q_1$, $Q_2$, \dots, $Q_n$, with complete, balanced Pythagorean gain graph $\Psi = (K_n,0)^\eta$.
\end{lem}

\begin{proof} Let $Q_i$ have coordinate vector $a_i = (a_{i1},\ldots,a_{id})$ and $P$ have
coordinate vector $x$ in $\bbE^d$.  The tangent space at $Q_i^*$, revised by $\eta(i)$ to $T_i$, has equation $z = 2a_i \cdot x - c_i + \eta(i)$, where $c_i = a_i\cdot a_i$, the $z$-coordinate of $Q_i^*$.  Eliminating $z$, $T_i \cap T_j$ satisfies the equation 
$$
- 2(a_j - a_i) \cdot x + c_j - c_i = \eta(j) - \eta(i),
$$
which (with $z = 0$) defines $h_{ij}$, the projection into $\bbE^d$.

Let us calculate the Pythogorean coordinate of a point $P \in \bbE^d$.  It is 
$\psi_{ij}(P) = \|x-a_j\|^2 - \|x-a_i\|^2 = 2(a_i-a_j) \cdot x + c_j - c_i$.  Thus the equation of $h_{ij}$ can be written $\psi_{ij}(P) = \eta(j) - \eta(i)$.  Since the right side is a constant, $h_{ij}$ is a hyperplane perpendicular
to $Q_iQ_j$.  The form of the constant demonstrates that $\Psi = (K_n,0)^\eta$.
\end{proof}

We note that Lemma \ref{L1} is a trivial generalization of \cite[Lemma 4]{Auren}.

\begin{prop} \mylabel{P1} 
(a) Let $T_1, T_2, \hdots, T_n$ be arbitrary nonvertical affine $d$-spaces in 
$\bbE^{d+1}$, no two parallel, and let $h_{ij}$ be the projection into $\bbE^d$ of $T_i \cap T_j$.  Then $\{ h_{ij} \}$ is a Pythagorean arrangement described by a balanced, complete gain graph $\Psi$, based on those points $Q_1, Q_2, \hdots, Q_n$ such that $Q_i^*$ is the unique point where a translate of $T_i$ is tangent to $S$.

(b) Conversely let $\Phi$ be a balanced, complete gain graph on $n$ vertices and let 
$Q_1, Q_2, \hdots, Q_n$ be distinct points in $\bbE^d$.  Then there exist affine $d$-spaces $T_1, T_2, \hdots, T_n$ in $\bbE^{d+1}$ such that the Pythagorean arrangement $\cH(\Phi; Q_1, \hdots, Q_n)$ equals the arrangement $\{ h_{ij} \}$ derived from $T_1, T_2, \hdots, T_n$ by the procedure in (a).
\end{prop}

\begin{proof}  
(a)  Suppose $T_i$ has to be lowered vertically a distance $\eta(i)$ to become tangent to $S$.  Then $\{ h_{ij} \}$ is as defined before Lemma \ref{L1} and $\Psi$ is as in Lemma \ref{L1}.

(b)  Since $\Psi$ is balanced and connected, there exist numbers $\eta(1), \eta(2), \ldots, \eta(n)$, unique up to an additive constant, so that $\psi(e; i, j) = \eta(j) - \eta(i)$.  (This is a well-known characterization of tensions on a graph in terms of potentials.  See, e.g., \cite{Berge1} or \cite[\S 2.3, Thm.\ 5]{Berge2}.)  Thus $\Psi = (K_n,0)^\eta$.  The $d$-space $T_i$ is the result of raising $T_{Q^*_i}$ the distance $\eta(i)$.
\end{proof}


\section{Invariants and face enumeration} \mylabel{invar} 

There is a class of invariants of hyperplane arrangements that are determined by the geometric semilattice:  they include the Whitney numbers and the characteristic polynomial.
These invariants of a Pythagorean arrangement $\cH(\Phi;\bQ)$ in $\bbE^d$ with generic reference points are readily derivable from corresponding invariants of $\Phi$, because by Corollary \ref{C1} $\Lat\cH(\Phi;\bQ) \cong [\Latb\Phi]^d_0$.  Similarly, if $\alpha \neq 0$, one can deduce the invariants of a generic non-Pythagorean arrangement $\cH(\alpha,\Phi;\bQ)$ from the fact that $\Lat \cH(\alpha,\Phi;\bQ) \cong [\Latb\Psi]^d_0$, where $\Psi$ is as in Section \ref{nonpyth}.

Throughout this section, $\Phi$ is a real, additive gain graph with $n$ vertices and $\bQ = (Q_1,\hdots,Q_n) \in (\bbE^d)^n$.

Reasonably obvious is this basic observation:

\begin{thm} \mylabel{P2} 
Given $\Phi$, $d$, and $\alpha$ (zero or not), the numbers of $k$-dimensional flats, faces, and bounded faces of $\cH(\alpha,\Phi;\bQ)$ are maximized for generic $\bQ$; specifically, when $\bQ$ has general position with respect to $\Phi$.
\hfill $\square$
\end{thm}

The fundamental enumerative result about arrangements of perpendiculars is

\begin{thm} \mylabel{Tenum} 
Given $\Phi$, $d$, and $\alpha$ (zero or not).  Generically, the numbers of $k$-dimensional flats, faces, and bounded faces of $\cH = \cH(\alpha,\Phi;\bQ)$ are given by the formulas
\begin{align*}
f_k(\cH) &= {\sum^d_{j=d-k}} | w_{d-k,j}(\Latb \Psi) | ,  \\
b_k(\cH) &= | {\sum^d_{j=d-k}} w_{d-k,j}(\Latb \Psi) | ,  \\
a_k(\cH) &= W_{d-k}(\Latb \Psi)
\end{align*}
for $0 \leq k \leq d$, where $\Psi$ is the Pythagorean gain graph of $\cH$, defined by \eqref{Epsi}.  The numbers of flats and faces of $\cH_\bbP$ if $0 \leq k \leq d < n$ are given by
\begin{align*}
f_k(\cH_\bbP) &= | \sum_{\substack{ j=d-k \\ d-j\text{ even} }}^d  w_{d-k,j}(\Lat L_0(\Psi)) | , \\
a_k(\cH_\bbP) &= W_{d-k}(\Lat L_0(\Psi)).
\end{align*}
\end{thm}

\begin{proof}
This is merely a combination of the general arrangement enumerations in Section \ref{arrs} with the particular description of generic $\cL(\cH)$ in Section \ref{gp}.
\end{proof}

In order to state these results about $\cH$ most elegantly we need to define some
polynomials.  The {\it characteristic polynomial} of an affine arrangement $\cE$ in $\bbE^d$ is 
\begin{equation*}
p_\cE (\lambda) = \sum^d_{j=0} w_j(\cE)\lambda^{d-j}.
\end{equation*}
The {\it Whitney-number polynomial} is 
\begin{equation*}
w_\cE(x,\lambda) = \underset {0\leq i\leq j\leq d}{\sum\sum} w_{ij}(\cE)x^i\lambda^{d-j}.
\end{equation*}
(In these formulas we define $w_j = w_{ij} = 0$ if $j > \rk \cE$.) Thus \eqref{first-num-def} and \eqref{face-num-def} can be expressed by the formulas
\begin{equation*}
f_d(\cE) = (-1)^dp_\cE(-1) \quad\text{ and }\quad \sum_k f_k(\cE)x^k = (-1)^dw_\cE(-x,-1),
\end{equation*}
while
\begin{equation*}
b_d(\cE) = (-1)^dp_\cE(1)\quad \text{ and }\quad \sum_kb_k(\cE)x^k = (-1)^dw_\cE(-x,1).
\end{equation*}

Turning to a gain graph $\Phi$ of order $n$, we may define its {\it balanced chromatic polynomial} as 
\begin{equation*}
\chi^\textb_\Phi(\lambda) = \sum^n_{j=0} w_j(\Latb\Phi)\lambda^{n-j}
\end{equation*}
and its {\it balanced Whitney-number polynomial} as 
\begin{equation*}
w^\textb_\Phi(x,\lambda) = \underset {0\leq i\leq j\leq n}{\sum\sum} w_{ij}(\Latb\Phi)
x^i\lambda^{n-j}.
\end{equation*}
(Again, $w_j = w_{ij} = 0$ if $j > \rk(\Latb\Phi)$.)  If $\Phi$ has $c$ connected components (thus $\Latb\Phi$ has rank $n-c$), the \emph{characteristic} and \emph{Whitney-number polynomials} of $\Latb\Phi$ are 
\begin{equation*}
p_{\Latb\Phi}(\lambda) = \lambda^{-c}\chi^\textb_\Phi(\lambda)\quad 
\text{ and }\quad w_{\Latb\Phi}(x,\lambda) = \lambda^{-c}w^\textb_\Phi(x,\lambda).
\end{equation*}

\begin{thm} \mylabel{Tinvarp}
Let $\cH = \cH(\Phi;\bQ)$ have generic reference points in $\bbE^d$, where $d\leq n$. Then $p_\cH(\lambda)$ and $w_\cH(\lambda)$ equal the polynomial parts of
$\chi^\textb_\Phi(\lambda)/\lambda^{n-d}$ and $w^\textb_\Phi(x,\lambda)/\lambda^{n-d}$, respectively.
\end{thm}

\begin{proof}  
This is the conclusion of the preceding discussion.
\end{proof}

Theorems \ref{Tenum} and \ref{Tinvarp} reduce the problem of counting regions or faces in $\cH$ to that of finding the balanced chromatic or Whitney-number polynomial of $\Phi$.  This approach will be illustrated in Section \ref{enum}.  

For non-Pythagorean descriptors there is a similar result.

\begin{thm} \mylabel{Tinvarnp}
With notation as in Corollary \ref{C2}, let $\cH = \cH(\alpha,\Phi;\bQ)$ with generic
$\bQ \in (\bbE^d)^n$.  Then $p_\cH(\lambda)$ and $w_\cH(x,\lambda)$ are the polynomial parts of $\chi^\textb_\Psi(\lambda)/\lambda^{n-d}$ and $w^\textb_\Psi(x,\lambda)/\lambda^{n-d}$, respectively. 
\hfill$\square$
\end{thm}

\begin{exam} \mylabel{Xdeforminvar}
For the affinographic arrangements of Example \ref{Xdeform}, $\Lat
\cH(\Phi;\bQ)\cong \Latb\Phi$.  (The reference points are generic by Proposition \ref{Ppoint}.)  Thus Theorem \ref{Tinvarp} applies with $d = n$.
\end{exam}
\medskip

The Pythagorean planar case has a nice description in terms of the structure of $\Phi$.  Let $m_{ij}$ be the number of edges between $i$ and $j$ in $\Phi$, $q$ the total number of edges, and $s_2$ the second elementary symmetric function of the $m_{ij}$.  Let $t$ be the number of balanced triangles in $\Phi$ and let $t_0$ be the number of triangles in the zero-gain edge set.

\begin{cor} \mylabel{T7} 
If $Q_1, \hdots, Q_n$ are generic in $\bbE^2$, then the Pythagorean arrangement of lines $\cH(\Phi; Q_1, \hdots, Q_n)$ has 
\begin{equation*}
\begin{array}{rll}
  a_2 =& q                       &\text{ lines},                     \\
  a_1 =& s_2                  &\text{ points},                    \\
  f_2 =& 1 + q + s_2 - t      &\text{ regions},                   \\
  f_1 =& q + 2s_2 - 3t        &\text{ geometric edges (1-faces)}, \\
  b_2 =& 1 + s_2 -q - t       &\text{ bounded regions},           \\
  b_1 =& 2s_2 - q - 3t        &\text{ bounded geometric edges.} 
\end{array}
\end{equation*}
\end{cor}

\begin{cor} \mylabel{T8} 
Let $\alpha$ be a nonzero real number.  If $Q_1, \hdots, Q_n$
are generic in $\bbE^2$, then the arrangement of
perpendicular lines $\cH(\alpha, \Phi ; Q_1, \hdots, Q_n)$ has numbers as in
Corollary \ref{T7} with $t$ replaced by $t_0$.  
\end{cor}

\begin{proof} Corollary \ref{T7} is obtained from Theorems \ref{T1} and \ref{Tenum} by way of Lemma \ref{Lwh}.  Corollary \ref{T8} follows similarly from Theorems \ref{T2} and \ref{Tenum}.
\end{proof}

Examples \ref{Xbal4}, \ref{Xcontra3}, and \ref{Xcontra03} illustrate these corollaries.

\begin{lem} \mylabel{Lwh}
Suppose a finite gain graph $\Phi$ has $q$ edges, all of them links (two distinct endpoints), and $t$ balanced triangles and $s_2$ is the second elementary symmetric function of the edge multiplicities.  Then in $\Latb \Phi$ the first few Whitney numbers are
\begin{alignat*}{4}
w_0 =\ &w_{00} = W_0 = 1, &\qquad &w_1 = w_{01} = -q, &\qquad w_2 = w_{02} = s_2 - t, \\
&w_{11} = W_{1} = q, &\qquad &w_{22} = W_2 = s_2 - 2t, &\qquad w_{12} = 2s_2 - 3t.
\end{alignat*}
If $t=0$ and there are $F_{n-3}$ forests with three edges and $t'$ balanced quadrilaterals, then
\begin{alignat*}{3}
w_3 = w_{03} &= -(F_{n-3}-t'), &\qquad  &w_{13} = F_{n-3}-4t', \\
w_{23} &= -(F_{n-3}-6t'), &\qquad W_3 =\ &w_{33} = F_{n-3}-3t'.
\end{alignat*}
\end{lem}

\begin{proof} 
The essential fact for $w_{i2}$ is that there are two kinds of balanced flat of rank 2: a balanced triangle, and a pair of nonparallel edges not contained in a balanced triangle.  There are $t$ of the former type and $s_2 - 3t$ of the latter.  A similar remark applies to the calculation of $w_{i3}$.
\end{proof}


\section{Enumeration in examples} \mylabel{enum}

We now have the machinery to do Pythagorean arrangements of special kinds.  
Here we describe some abstractly; then in the next section we see how they
may arise from various models of voter preference.  Throughout, $\Phi$ is a real, additive gain graph with $n$ vertices and $\bQ = (Q_1,\hdots,Q_n) \in (\bbE^d)^n$.

\subsection{Perpendicular bisectors, power-diagram arrangements, and other balanced gain graphs.} \mylabel{X1}
Suppose the arrangement $\cH$ consists of one perpendicular hyperplane for each line $Q_iQ_j$, positioned according to a balanced Pythagorean gain graph $\Phi$.  We obtain the theorem of Good and Tideman as well as generalizations to arrangements based on power diagrams (see Section \ref{switch}) and to a voter with prior biases (as at the end of Section \ref{voter}).

\begin{cor} \mylabel{T3} 
Let $\Phi$ be a balanced, complete Pythagorean gain graph on $n$ vertices.  Let $\cH$ be the Pythagorean arrangement $\cH(\Phi;\bQ)$.  Generically, the
numbers of $k$-dimensional flats, faces, and bounded faces of $\cH$ are, for regions,
\begin{equation*}
\begin{aligned}
f_d(\cH) &= {{\sum_{i=0}^d}} |s(n,n-i)|, \qquad
b_d(\cH) &= (-1)^d {{\sum_{i=0}^d}} s(n, n-i), 
\end{aligned}
\end{equation*}
and for the rest,
\begin{equation*}
\begin{aligned}
a_k(\cH) &= S(n,n-d+k), \\
f_k(\cH) &= S(n,n-d+k) {{\sum^d_{j=d-k}}} |s(n-d+k, n-j)|, \\
b_k(\cH) &= (-1)^k S(n,n-d+k) {{\sum^d_{j=d-k}}} s(n-d+k, n-j) . 
\end{aligned}
\end{equation*}
\end{cor}

\begin{proof} We combine Theorem \ref{Tenum} with standard facts about $\Pi_n$, the set of partitions of $n$ elements ordered by refinement, in which $\rk\pi = n - |\pi|$.  Since every edge set in $\Phi$ is balanced, $\Latb\Phi = \Lat G(K_n) = \Pi_n$.  The Whitney numbers of the first and second kinds of $\Pi_n$ are the Stirling numbers (see \cite{F-Rota} or \cite[\S 9]{FCT} for the first kind, \cite{RotaP} for the second).  The doubly indexed Whitney numbers are $w_{ij}(\Pi_n) = S(n, n-i) s(n-i, n-j)$.
\end{proof}

The planar numbers result from substituting $q = \binom{n}{2}$, $s_2 = \binom{q}{2} = 3 \binom{n+1}{4}$, and $t = \binom{n}{3}$ in Corollary \ref{T7}.

If $\Phi$ is balanced but incomplete, the conclusion is more complicated.
A partial result is easiest to state.  It can be understood as an application of Theorem \ref{Tinvarp}.

\begin{cor} \mylabel{T4}
Let $\Phi$ be a balanced Pythagorean gain graph on $n$ vertices and let $\chi_\Gamma(\lambda)$ be the chromatic polynomial of the underlying graph $\Gamma$.  Generically, $p_{\cH(\Phi;\bQ)}(\lambda)$ is the polynomial part of $\chi(\lambda)/\lambda^{n-d}$.  
The generic numbers of regions and bounded regions of $\cH(\Phi;\bQ)$ are
\begin{enumerate}
\item[] $f_d =$ the sum of magnitudes of the $d+1$ leading coefficients of
$\chi_\Gamma(\lambda)$, and
\item[] $b_d =$ the magnitude of the sum of the $d+1$ leading coefficients.
\end{enumerate}
\end{cor}

\begin{proof} Rota \cite[\S 9]{FCT} proved that 
$\chi_\Gamma(\lambda) = \sum w_i (\Lat\Gamma)\, \lambda^{n-i}$.  We know $\Lat \Gamma = \Latb \Phi$ from balance of $\Phi$.
\end{proof}

In the planar case we may again resort to Corollary \ref{T7}, as when $\Phi$ was complete: still $s_2 = \binom{q}{2}$, but now $q$ and $t$ depend on $\Phi$.

\begin{figure} \mylabel{Fbal4}
\vbox to 3.7truein{}
\begin{center}
\includegraphics{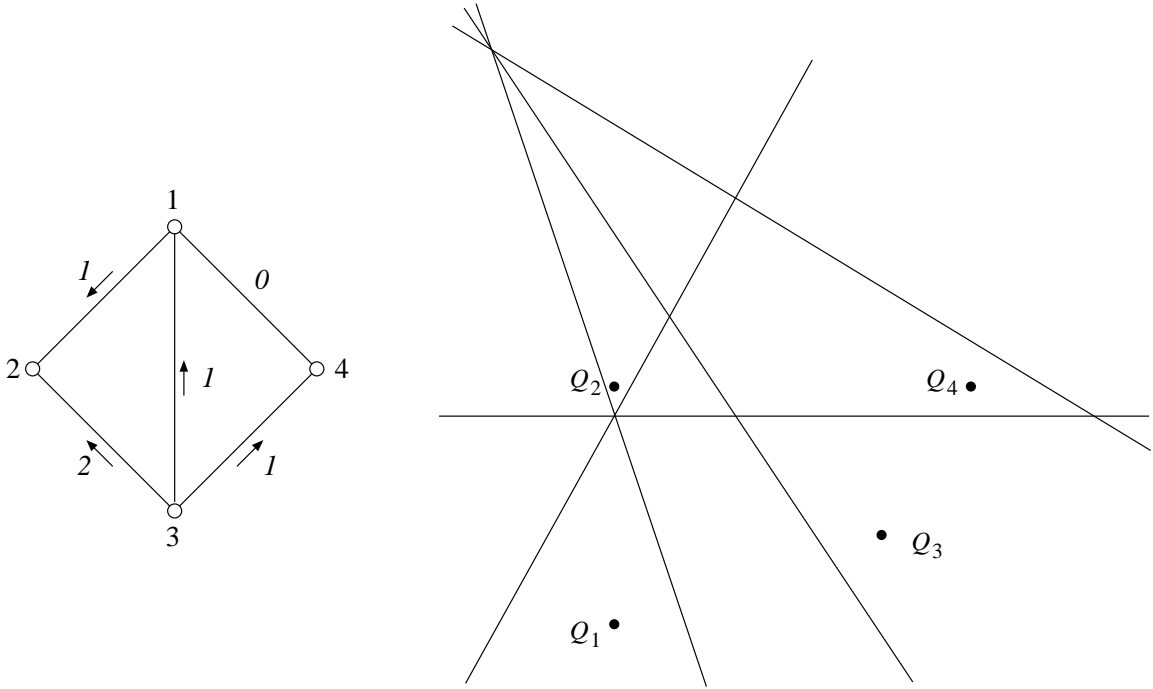}
\end{center}
\caption{A balanced gain graph and a corresponding generic line arrangement.  (Example \ref{Xbal4}.)}
\end{figure}

\begin{exam} \mylabel{Xbal4}
In Figure \ref{Fbal4} we see a balanced gain graph and a corresponding generic Pythagorean arrangement of lines.  Since $\Phi$ is balanced, $\chi_\Phi(\lambda) = \chi_\Gamma(\lambda) = \lambda(\lambda-1)(\lambda-2)^2$, $\Gamma$ being the underlying graph of $\Phi$.  By Theorem \ref{Tinvarp} therefore $p_\cH(\lambda)$ is the polynomial part of $\lambda^{-2}(\lambda^4-5\lambda^3+8\lambda^2-4\lambda)$, so that $p_\cH(\lambda) = \lambda^2-5\lambda+8$.  It follows that the number of regions is $|p_\cH(-1)| = 14$ and the number of bounded regions is $|p_\cH(1)| = 4$, in agreement with the diagram.
\placefigure{bal4}
\end{exam}

\subsection{Generic hyperplanes and non-Pythagorean rules.} \mylabel{X2}
(\emph{Forests.})
Suppose $\Phi$ is a Pythagorean gain graph in which no circle is balanced, as would (almost surely) happen if the gains were chosen at random, or suppose $(\alpha,\Phi)$ is a non-Pythagorean descriptor (i.e., $\alpha \neq 0$) in which the zero-gain edges of $\Phi$ contain no circle.  These two structures have the same formulas.  In either case there are no balanced circles in the appropriate biased
graph, so the balanced flats of the lift are precisely the spanning
forests of the underlying graph $\Gamma$.  Let $F_{l}(\Gamma)$ denote
the number of spanning forests with $l$ components.

\begin{cor} \mylabel{T5}
Let $\Phi$ be a Pythagorean gain graph without balanced circles and let 
$\alpha = 0$, or let $(\alpha,\Phi)$ be a non-Pythagorean descriptor without
zero-gain circles.  Let $\Gamma$ be the underlying graph.  Generically, $p_{\cH(\alpha,\Phi;\bQ)}(\lambda) = \sum_{i=0}^d (-1)^i F_{n-i}(\Gamma) \lambda^{d-i}$.
The numbers
of flats and faces of $\cH(\alpha,\Phi;\bQ)$ when $\bQ$ is generic are
\begin{equation*}
\begin{aligned}
f_d &= {{\sum^d_{i=0}}} F_{n-i}(\Gamma), \qquad 
    b_d = {{\sum^d_{i=0}}} (-1)^{d-i} F_{n-i}(\Gamma), \\
a_k &= F_{n-d+k} (\Gamma), \qquad \
    f_k = {{\sum^d_{i=d-k}}} {\binom{i}{d-k}} F_{n-i}(\Gamma), \\
b_k &= {{\sum^d_{i=d-k}}} (-1)^{d-i} {\binom{i}{d-k}} F_{n-i}(\Gamma) 
    \qquad \text{ (except when } k = d - n + c(\Phi) > 0).
\end{aligned}
\end{equation*}
\end{cor}

\begin{proof}[Outline of proof]
The Pythagorean gain graph $\Psi$ of $\cH$ (which equals $\Phi$ if $\alpha=0$) is contrabalanced (see \cite[Ex.\ III.3.4]{BG}) so $\chi^\textb_\Psi(\lambda) = \sum (-1)^i F_{n-i}(\Gamma) \lambda^{d-i}$.  Apply Theorem \ref{Tinvarnp} to get $p_\cH(\lambda)$.

The fact lying behind the counts, that the doubly indexed Whitney numbers of $\Latb \Psi$ are binomial multiples of the forest numbers of $\Gamma$, is implicit in \cite[Thm.\ 7]{BGLF} and explicit in \cite[Ex.\ III.5.4]{BG}.  This fact was recently rediscovered in \cite[\S 5]{P-S} in the language of what in \cite[\S IV.4]{BG} we would call the canonical affine hyperplanar lift representation of $\Psi$.
\end{proof}

The first three forest numbers are simple.  If $\Phi$ has $q$ edges, then 
$F_n(\Gamma)=1$, $F_{n-1}(\Gamma)=q$, $F_{n-2}(\Gamma) = \binom{q}{2}$.
One can now write down explicit formulas for $a_k, f_k$, and $b_k$ in
dimensions $d=1$ and $2$.  In $\bbE^2$, $f_2 = 1 + \binom{q+1}{2}$, 
$f_1 = q^2$, $b_0 = f_0 = \binom{q}{2}$, $b_2 = \binom{q-1}{2}$, 
$b_1 = q^2 - 2q$.  (We may obtain these formulas also from Corollary \ref{T8} with $t=t_0=0$.)

\begin{exam} \mylabel{X2a}
\emph{Equally many perpendiculars to each line.}
Take $\Gamma = mK_n$, a complete graph with $m$ edges between each pair of vertices.  Then
\begin{equation*}
F_{n-i}(mK_n) = \frac{m^i}{(n-i)!} {\sum^{n-i}_{k=0}} (-{\tfrac 12})^k
\binom{n-i}{k} \binom{n-1}{i-k} (n-i+k)! n^{i-k}
\end{equation*}
by R{\'e}nyi's formula for $F_{n-i}(K_n)$ \cite{Renyi}.  Even without R{\'e}nyi's formula it is easy to see that 
$$
F_{n-1}(mK_n) = m \tbinom{n}{2}, \quad 
F_{n-2}(mK_n) = 3m^2 \tbinom{n+1}{4}, \quad 
F_{n-3}(mK_n) = m^3 \tbinom{n}{3} \frac{n^2-5n-12}{8}.
$$
\end{exam}

\begin{figure} \mylabel{Fcontra3}
\vbox to 5truein{}
\begin{center}
\includegraphics{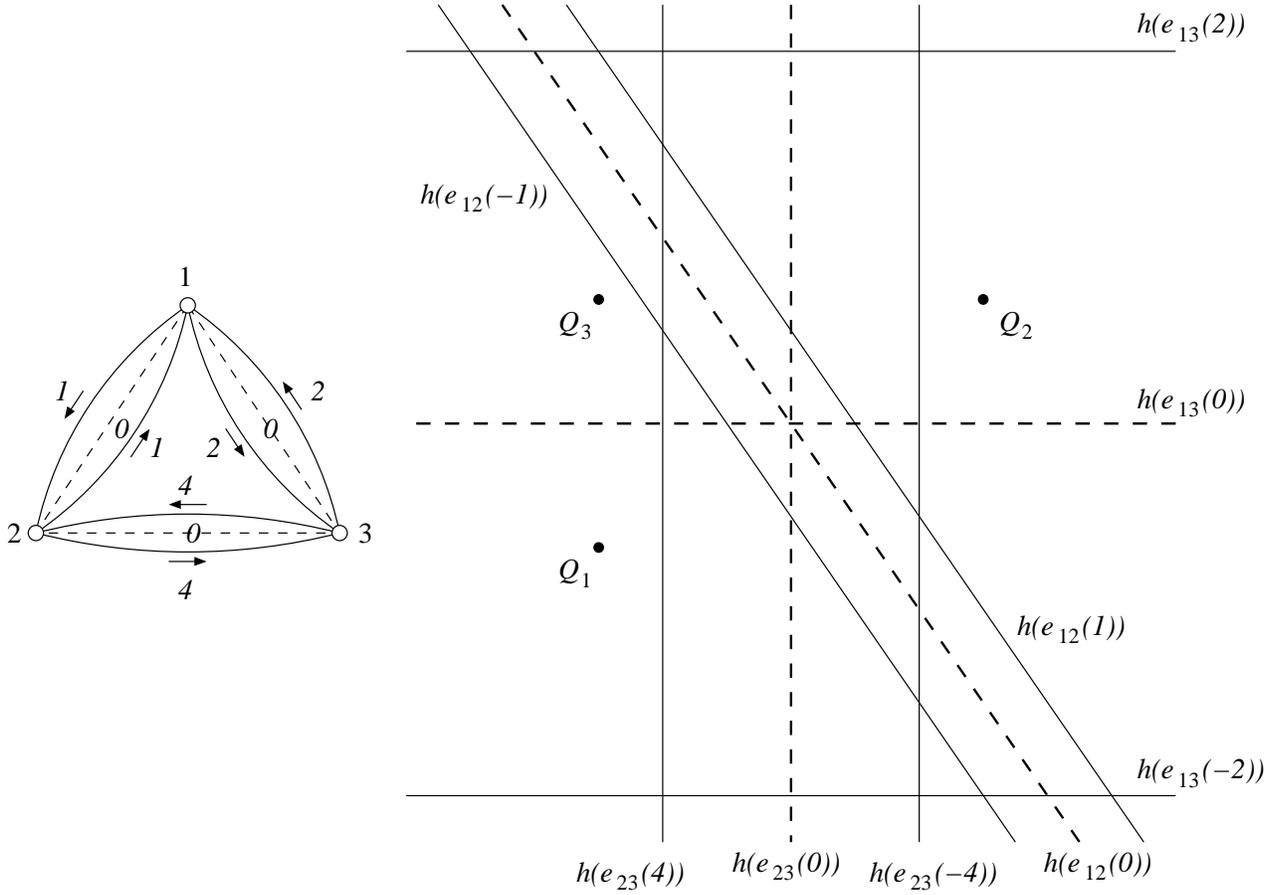}
\end{center}
\caption{Illustration of Examples \ref{Xcontra3} and \ref{Xcontra03}.  The solid edges and lines show generic gains and a corresponding generic Pythagorean line arrangent (Example \ref{Xcontra3}.)  With the dashed lines we have also all $0$-gain edges and, correspondingly, all perpendicular bisectors (Example \ref{Xcontra03}).}
\end{figure}

\begin{exam} \mylabel{Xcontra3}
The solid lines in Figure \ref{Fcontra3} show a gain graph, $\Phi$, without balanced circles and a generic Pythagorean arrangement of lines, $\cH(\Phi;\bQ)$.  Since each two vertices are doubly adjacent, we are in the situation of Example \ref{X2a} with $m=2$.  The characteristic polynomial is $p_\cH(\lambda) = \lambda^2-6\lambda+12$; there are 19 regions, of whcih 7 are bounded.
\placefigure{contra3}
\end{exam}

\subsection{Generic hyperplanes and non-Pythagorean rules, with all bisectors and with equally many perpendiculars to each line.} \mylabel{X3}  
(\emph{Fat forests.})
As in Section \ref{X2} this is two kinds of example in one treatment.  Here $\Phi$ has a complete balanced part as in Section \ref{X1} and a totally unbalanced part as in Section \ref{X2}.  Assume that each pair of vertices is joined by $M+1$ edges. If $\Phi$ is a Pythagorean gain graph, we assume that all balanced circles lie in a balanced, complete spanning subgraph; for instance, they may be the zero-gain edges.  If $\Phi$ is part of a descriptor $(\alpha, \Phi)$ with $\alpha \neq 0$, we assume that the set $Z$ of zero-gain edges forms a complete spanning subgraph $(V,Z)$.  (Thus in either case, as a biased graph \cite{BG}, $\Phi = \bgr{K_n} \cup (MK_n,\eset)$, where $\bgr{K_n}$ is a balanced $K_n$ and $(MK_n,\eset)$ is $K_n$ with each edge replaced by $M$ distinguishable copies of itself and with no balanced circles.)

The semilattice $\Latb \Phi$ is isomorphic to the lower $d$ ranks of the geometric semilattice of \emph{spanning fat forests of $MK_n$}.  A spanning fat forest $(\pi,F)$ consists of a partition $\pi$ of the vertex set together with an edge set $F \subseteq E(MK_n)$ such that, if each block of $\pi$ is collapsed to a point, $F$ contains no circle.  The geometric lattice of fat forests of a graph and related lattices and semilattices will be studied in detail in the anticipated \cite[Ch.\ IV]{GLSP}.  

We mention three results from \cite[Ch.\ IV]{GLSP}.  First, the top Whitney number of $\Latb \Phi$ is 
$$
w_{n-1} = (-1)^{n-1}\; \frac{(n-1)!}{nM}\, \sum_{\mu \vdash n} \frac{(nM)^{\mu_1+\mu_2+\cdots}}{\mu_1!\mu_2!\cdots};
$$
here $\mu \vdash n$ means that $\mu = (\mu_1,\mu_2,\hdots)$ with all $\mu_i \geq 0$ and $\sum j\mu_j = n$.  The other Whitney numbers are given by
$$
w_{n-i} = (-1)^{n-i}\; n! \sum_{\substack{\lambda \vdash n \\  \sum\lambda_k=i}} \prod_{k=1}^\infty \, \frac{1}{\lambda_k!} \left[ - \frac{1}{k^2M} \sum_{\mu \vdash k} \frac{(kM)^{\mu_1+\mu_2+\cdots}}{\mu_1!\mu_2!\cdots} \right]^{\lambda_k}.
$$
These formulas are not among the easiest but they do permit computation of the generic numbers of regions and bounded regions in the hyperplane arrangements $\cH$ of this example.  Second, if $\chi^\textb_{n,M}(\lambda)$ denotes the balanced chromatic polynomial of $\Phi$, then $\{\chi^\textb_{n,M}(\lambda)\}_{n=1}^\infty$ has exponential generating function
$$
-1 + \exp\left( -\frac{\lambda}{M} \sum_{k=1}^\infty \frac{(-z)^k}{k^2} \sum_{\mu\vdash k} \frac{(kM)^{\mu_1+\mu_2+\cdots}}{\mu_1!\mu_2!\cdots} \right).
$$
From this the region and bounded region numbers can, in principle, be obtained by the formulas of Section \ref{invar}.  
Third, the number $W_{n-i}$ of spanning fat forests in $MK_n$ with $i$ connected components (recall that $W_k = a_{d-k}$, the number of flats in $\cH$ of codimension $k$) is given by the same formula as for $w_{n-i}$ except with the signs omitted and an extra factor of $1!^{\mu_2}2!^{\mu_3}\cdots$ in the denominator of the inmost sum.  For instance, the number of spanning fat trees is
$$
W_{n-1} = \frac{n!}{n^2M} \sum_{\mu \vdash n} \prod_{j=1}^{\infty} \frac{(nM)^{\mu_j}}{{\mu_j}!(j-1)!^{\mu_j}}.
$$

We get much simpler evaluations in the planar case by applying  Corollaries \ref{T7} and \ref{T8}, in which $q = (M+1) \binom{n}{2}$, $s_2 = 3 (M+1)^2 \binom{n+1}{4}$, and $t = t_0 = \binom{n}{3}$.

\begin{exam} \mylabel{Xcontra03}
The solid and dashed lines in Figure \ref{Fcontra3} show a gain graph and generic Pythagorean arrangement of lines of the kind in Section \ref{X3}.  The balanced circles lie in the set $Z$ of edges with gain $0$.  Since $q=9$, $s_2=27$, and $t=1$, Corollary \ref{T7} says there are 36 regions of which 18 are bounded.
\end{exam}

\subsection{Symmetric, uniform Pythagorean hyperplanes: odd case.} \mylabel{Xodd}  
(\emph{Composed partitions.})
For positive integers $k$ and $n$, let $\Phi_n = [-k,k]K_n$ be the additive real gain graph on $n$ vertices that has an edge of gain $i$ between each pair of vertices for every $i = 0, \pm1, \hdots, \pm k$.  This gain graph, or any other obtained through multiplying all gains by a positive constant $\delta$, gives the odd number $2k+1$ of perpendiculars to each reference line, placed symmetrically about the bisector and equally spaced, and---in Pythagorean coordinates---identically spaced along all reference lines.

We can easily show that $[-k,k]K_n$ has balanced chromatic polynomial
\begin{equation} \mylabel{Ecompchr}
\chi^\textb_{[-k,k]K_n}(\lambda) = \lambda(\lambda-nk-1)_{n-1},
\end{equation}
where $(x)_r$ is the falling factorial $x(x-1)\cdots(x-r+1)$, whence $C_n(k) = \Latb [-k,k]K_n$ has characteristic polynomial
\begin{equation} \mylabel{Eodd}
p_{C_n(k)}(\lambda) = (\lambda-nk-1)_{n-1}.
\end{equation}
Applying Theorem \ref{Tinvarp} to Equation \eqref{Ecompchr}, in $\bbE^d$ the number of regions, $f_d$, will equal the sum of the magnitudes of the coefficients of $\lambda^{n}, \hdots, \lambda^{n-d}$ in $\chi^\textb_{[-k,k]K_n}(\lambda)$.  The number of bounded regions, $b_d$, will equal the magnitude of the sum of the same coefficients.  Furthermore, by \cite[Thm.\ III.5.2]{BG}, the characteristic polynomial of the complete lift is
\begin{equation} \mylabel{Eoddcomplete}
p_{\Lat L_0([-k,k]K_n)}(\lambda) = (\lambda-1) (\lambda-nk-1)_{n-1}.
\end{equation}

To prove \eqref{Eodd} we treat the gains modulo $N$, where $N > nk$, and employ gain graph coloring \cite[\S III.4]{BG}.  
Taking gains in $\bbZ_n$ does not change the balanced circles because the largest possible sum of gains around a circle is $nk$.  
A zero-free proper coloring of $\Phi$ is a function $c : [n] \to \bbZ_N$ such that $c(i)$ and $c(j)$ differ by more than $k$ whenever $i\neq j$.  
($[n]$ is $\{1,2,\hdots,n\}$.) 
We obtain every such function by choosing its image, which can be done in $\frac{N}{N-nk} \binom{N-nk}{n}$ ways by \cite[Formula {[9b]}]{Comt}, and then choosing the coloring with that image in any of $n!$ ways.  Thus there are $N(N-nk-1)_{n-1}$ proper colorings.
That is, with gain group $\bbZ_N$ we have balanced chromatic polynomial $\chi^*_N$
for which $\chi^*_N(N+1) = N(N-nk-1)_{n-1}$.  Since by \cite[Thms.\ III.4.2 and III.5.3]{BG} the balanced chromatic polynomial depends only on which circles are balanced and it equals (up to a factor of $\lambda$) the characteristic polynomial of $\Latb \Phi$, we have Equation \eqref{Eodd}.  (This proof is from \cite[Ch.\ III]{GLSP}.  Later others gave various proofs.  Apparently the first published was that of Edelman and Reiner, whose proof is by an induction table for their arrangement $\cA_{r,I}^{(l)}$ with $I=[-l,l]$ \cite[p.\ 321]{E-R}; this arrangement (with their $l, r$ = our $k, n$) is the canonical linear hyperplane representation of $L_0([-k,k]K_n)$ \cite[\S IV.4]{BG}, whose characteristic polynomial is given by \eqref{Eoddcomplete}.  Athanasiadis' proof in \cite[Thm.\ 5.1]{Ath} is essentially the same as ours.)

The geometric semilattice $C_n(k)$ and related geometric lattices are very interesting objects.  Here I will merely point out that an element of $C_n(k)$ can be regarded as a kind of structured partition I call a \emph{weakly composed partition of $[n]$} (more specifically, a \emph{$k$-composed partition}; \emph{strictly composed} if $k=1$).  To explain this we first define a \emph{$k$-composition} of a set $B$:  it is an ordered weak partition of $B$,
that is, a sequence $(S_0,S_1,S_2,\hdots)$ of pairwise disjoint sets whose
union is $B$, in which there is no consecutive subsequence of $k$ empty sets
except in the infinite terminal string of nulls; for normalization we also
require $S_0 \neq \eset$ unless all $S_i = \eset$.  If $k=1$ this is just
a \emph{(strict) composition} of $B$, i.e., an ordered partition
$(S_0,S_1,\hdots,S_l)$, with a terminal string of null sets attached for
notational consistency.  A $k$-composed partition of $[n]$ is a
partition of $[n]$ together with a $k$-composition of each block.  One can explicitly describe the refinement ordering, interval structure, and characteristic polynomial of $C_n(k)$, $\Lat L_0([-k,k]K_n)$ (which consists of all $k$-composed partitions and ordinary partitions of $[n]$), and other related lattices and semilattices in terms of $k$-composed partial partitions \cite[Ch.\ III]{GLSP}.

Recently Gill \cite{Gill,Gill2} has studied the weakly composed partition semilattice $C_n(k)$ in the guise of the intersection semilattice of the canonical affine hyperplanar lift representation of $[-k,k]K_n$.  
In particular, $p_\cH(\lambda) = \lambda^2 - (2k+1)\binom{n}{2}\lambda + \big[3k(k+1)n+\frac{3n-1}{4}\big]\binom{n}{3}$ generically.

\begin{figure} \mylabel{Fodd4}
\vbox to 7.5truein{}
\begin{center}
\includegraphics{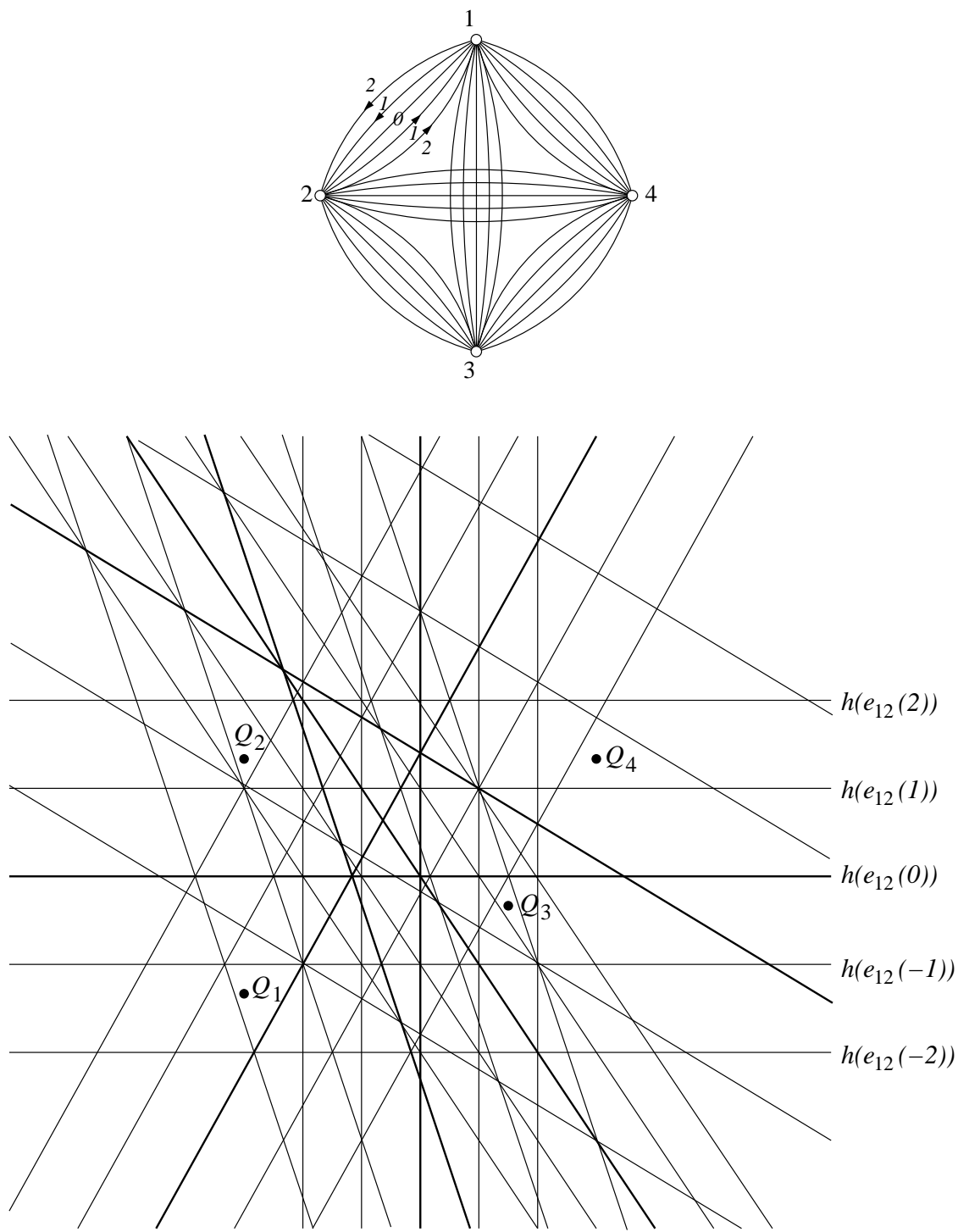}
\end{center}
\caption{The gain graph $\Phi_4 = [-2,2]K_4$ and a portion of the Pythagorean line arrangement of Example \ref{Xodd4}.  The heavy lines are the perpendicular bisectors.  In the gain graph, in each set of parallel edges the gains are the same; they are only marked for edges $e_{12}$.}
\end{figure}

\begin{exam} \mylabel{Xodd4}
Figure 11.3 displays $\Phi_4 = [-2,2]K_4$ and the generic planar arrangement $\cH(\Phi_4;\bQ)$ with $Q_1 = (0,0), Q_2 = (0,2), Q_3 = (2,1), Q_4 = (3,2)$.  Since $\chi^\textb_{\Phi_4}(\lambda) = \lambda(\lambda-9)_3$, the characteristic polynomial $p_\cH(\lambda) = \lambda^2-39\lambda+299$.  Thus there are 330 regions, 270 of them bounded.
\placefig{11.3}{Xodd4}
\end{exam} \smallskip

From Corollary \ref{T7} and Lemma \ref{Lwh} with $q = (2k+1)\binom{n}{2}$, $s_2 = 3(2k+1)^2\binom{n+1}{4}$, and $t = (3k^2+3k+1)\binom{n}{3}$ (or with greater difficulty from \eqref{Eodd}) we get the first few Whitney numbers of $C_n(k)$ and the face and flat numbers in two dimensions.

\subsection{The same, without bisectors.} \mylabel{Xoddnobi}
We take $\Phi'_n = \{\pm1,\pm2,\dots,\pm k\}K_n$ to be the additive real gain graph on $n$ vertices that has an edge of gain $i$ between each pair of vertices for every $i = \pm1, \hdots, \pm k$.  That is, it is the example of Section \ref{Xodd} without the bisectors of the reference segments.  This Pythagorean gain graph gives us $2k$ perpendiculars to each reference line, placed symmetrically about the bisector and equally spaced on each side of it, identically spaced---in Pythagorean coordinates---along all reference lines, but omitting the bisector.

The balanced chromatic polynomial, $\chi^\textb_{\Phi'_n}(\lambda)$, can be computed from \eqref{Ecompchr} by \cite[Prop.\ I.4.4]{GLSP}.  Omitting the details, which will appear in \cite[Ch.\ III]{GLSP}, the result is that
\begin{equation} \mylabel{Eoddnobi}
\sum_{n=1}^\infty \chi^\textb_{\Phi'_n}(\lambda) \frac{z^n}{n!} = e^{\lambda f(z)} - 1,
\end{equation}
where $f(z)$ is determined by
\begin{equation*} 
f'(z) = \sum_{j=0}^\infty \tbinom{\lambda-jk-1}{j} z^j \quad \text{ and } \quad f(0) = 0.
\end{equation*}
From this one can extract the balanced chromatic polynomial itself and thus the Whitney numbers of the first kind and, by the methods of Section \ref{invar}, the region and bounded region numbers in any dimension $d$.  

Replacing $z$ by $-z$, if we set $\lambda = -1$ Equation \eqref{Eoddnobi} becomes the exponential generating function for $r_n$, the number of regions of $\cH(\Phi'_n;\bQ)$ when $d=n-1$.  If we set $\lambda = 1$ it becomes the exponential generating function for the number of bounded regions.

In the plane we can compute all Whitney number from Lemma \ref{Lwh} and the values $q=2k\binom{n}{2}$, $s_2=12k^2\binom{n+1}{4}$, and $t=k(k+1)\binom{n}{3}$.  Then $p_\cH(\lambda) = \lambda^2 - 2k\binom{n}{2}\lambda + 3k(kn+1)\binom{n}{3}$.

\begin{figure} \mylabel{Foddnobi4}
\vbox to 7.5truein{}
\begin{center}
\includegraphics{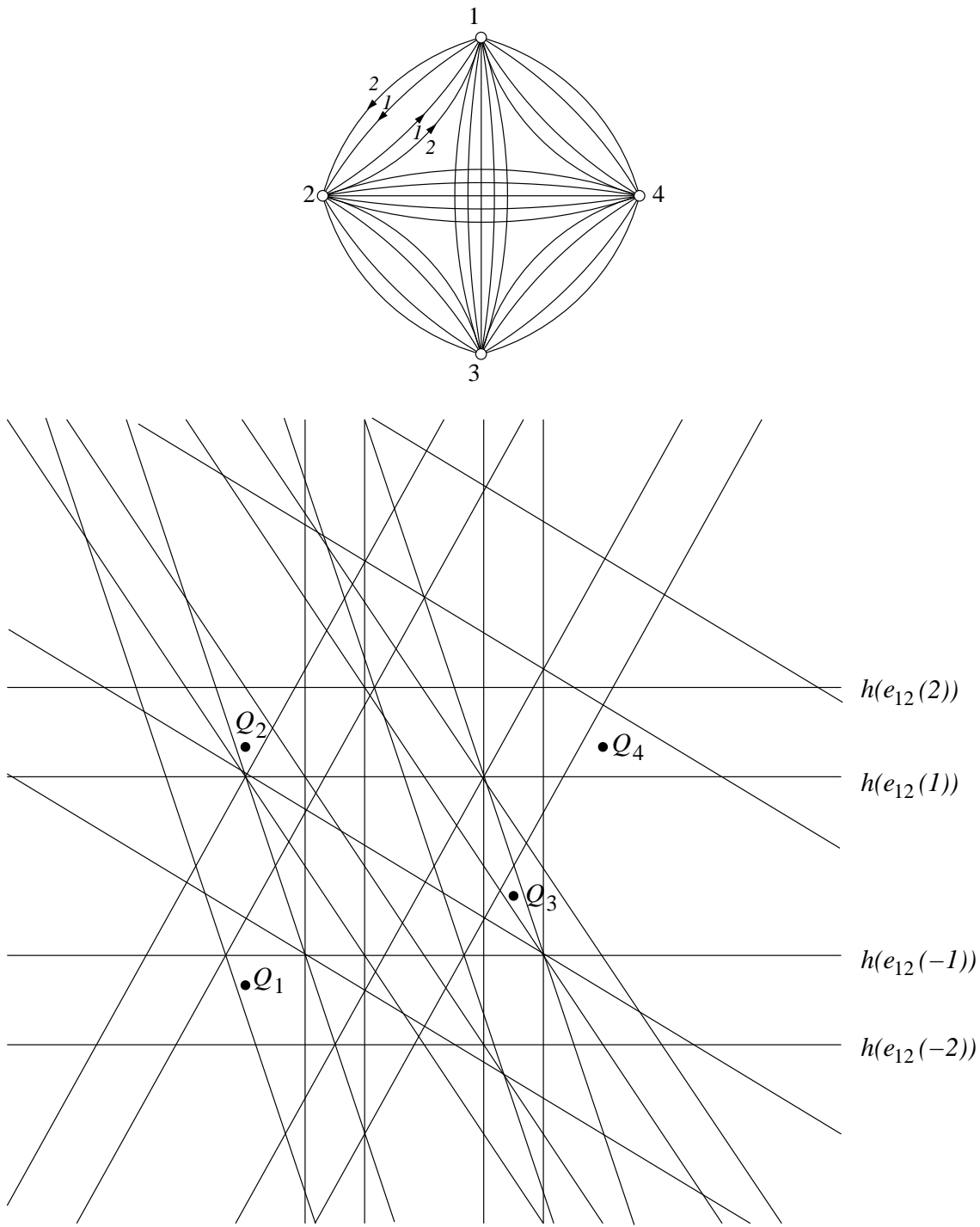}
\end{center}
\caption{The gain graph $\Phi'_4 = \{\pm1,\pm2\}K_4$ and part of the Pythagorean line arrangement from Example \ref{Xoddnobi4}.  The heavy lines are the perpendicular bisectors.}
\end{figure}

\begin{exam} \mylabel{Xoddnobi4}
In Figure 11.4 are $\Phi'_4 = \{\pm1,\pm2\}K_4$ and its Pythagorean line arrangement $\cH(\Phi'_4;\bQ)$.  From the general planar characteristic poynomial we have $p_\cH(\lambda) = \lambda^2 - 24\lambda + 216$.  There are 241 regions and 193 are bounded.
\placefig{11.4}{Xoddnobi4}
\end{exam}

\begin{exam} \mylabel{Xoddnobi1}
\emph{Just two perpendiculars to each reference line.}
We specialize to the case $k=1$, that is, 
$\Phi'_n = \{+1,-1\}K_n$.  We may look upon $\Phi'_n$ as a directed graph, an edge $e$ with $\phi'_n(e;i,j) = +1$ being interpreted as an arc directed from
$i$ to $j$.  Thus we are discussing the poise lift matroid of the complete digraph.  (See the definition of poise bias in \cite[Ex.\ I.6.5]{BG}.)  

Here the exponential generating function of $r_n$ has a remarkable expression.  The Catalan numbers $C_n = \frac{1}{n+1}\binom{2n}{n}$ have ordinary generating function $f_C(z) = \sum_0^\infty C_n z^n = \big(1-\sqrt{1-4z}\big)/2z$, as is well known.  Then $f(-z)\big|_{\lambda=-1} = f_C(z)-1$ so
$$
1 + \sum_{n=1}^\infty r_n \frac{z^n}{n!} = e^{f_C(z)-1} = 
\exp\left( \frac{1-\sqrt{1-4z}}{2z} - 1 \right).
$$
The explanation is found in the observation of Postnikov and Stanley that (in our language) $\cH([-1,1]K_n;\bQ)$ with the affinographic reference points of Example \ref{Xdeform}---equivalently, any affinely independent reference points---has $n! C_n$ regions \cite[Prop.\ 7.2]{P-S}.  Thus they call it and $\cH(\{\pm1\}K_n;\bQ)$ with similar reference points \emph{Catalan arrangements} \cite[(3.8)]{P-S}.

We mention that Athanasiadis \cite[Thm.\ 5.3]{Ath} has a different formula from ours for converting the exponential generating function of $p(\Latb [-1,1]K_n; \lambda)$ to that for $p(\Latb\Phi'_n; \lambda)$, as do Postnikov and Stanley though stated there only for the number of regions (\cite[Thm.\ 7.1]{P-S}, announced in \cite[Thm.\ 2.3]{Stan}.  The application to the characteristic polynomial follows from the exponential formula of \cite[Thm.\ 1.2]{Stan}).
\end{exam}

\subsection{Symmetric, uniform Pythagorean hyperplanes: even case.}  \mylabel{Xeven}
Here the Pythagorean gain graph is 
$\Phi''_n = \{\pm1,\pm3,\hdots,\pm(2k-1)\}K_n$, in which each
pair of vertices is joined by $2k$ edges whose gain values are $\pm i$ for every odd $i \leq 2k$ (or, what is essentially the same, $\pm\delta i$ for any fixed $\delta > 0$).  This gain graph corresponds to an arrangement of $2k$ perpendiculars to each reference line, symmetrically placed around the bisector and---in Pythagorean coordinates---identically spaced along all reference lines.

Athanasiadis \cite[Thm.\ 5.2]{Ath} has produced a remarkable expression for $p_{\Latb\Phi}(\lambda)$ (that is, $\chi^\textb_{\Phi''_n}(\lambda)/\lambda$) when $\Phi$ has a form that includes our examples from Sections \ref{Xodd}, \ref{Xoddnobi}, and \ref{Xeven}.  For $A$ a finite set of positive integers and $A^0 = A \cup \{0\}$, let $\Phi_n = (\pm A)K_n$ and $\Phi^0_n = (\pm A^0)K_n$.  By taking $A = \{1,3,\dots,2k-1\}$ in Athanasiadis' theorem we find that, for sufficiently large integers $\lambda$, $p(\Latb{\Phi''_n}^0; \lambda)$ is the coefficient of $x^{\lambda-2n}$ in the expression
$$
(n-1)! \left( \frac{1+x^{2k-1}}{1-x^2} \right)^n.
$$
The conclusion is that $\Latb {\Phi''_n}^0$ has characteristic polynomial
\begin{alignat*}{1}
{p''_n}^0(\lambda) &= \sum_{\substack{ i=0 \\ \text{even} }}^n \binom{n}{i} \left( \tfrac{\lambda+i}{2} - 1 - ki \right)_{n-1}   \\
                   &= \sum_{\substack{ i=0 \\ \text{odd} }}^n \binom{n}{i} \left( \tfrac{\lambda+i}{2} - 1 - ki \right)_{n-1}
\end{alignat*}
and therefore
$$
{p''_n}^0(\lambda) = \tfrac12 \sum_{i=0}^n \binom{n}{i} \left( \tfrac{\lambda+i}{2} - 1 - ki \right)_{n-1}
$$
for $n>0$.  Now we apply the argument of Athanasiadis \cite[Thm.\ 5.3]{Ath} to obtain the polynomial of $\Latb \Phi''_n$:
\begin{alignat*}{1}
p''_n(\lambda) &= \sum_{l=1}^n S(n,l) {p''_l}^0(\lambda)  \\
               &= \tfrac12 \sum_{l=1}^n \sum_{i=0}^l S(n,l) \binom{l}{i} \left( \tfrac{\lambda+i}{2} - 1 - ki \right)_{l-1}
\end{alignat*}
for $n>0$.  From this the characteristic polynomial of a generic arrangement $\cH(\Phi;\bQ)$ in $\bbE^d$ can be found via Theorem \ref{Tinvarp}.

For applications probably the most significant case is that in which $k=1$.  This case is easier to analyze because it is the same as Example \ref{Xoddnobi1}.

A triangle cannot be balanced in this example, no matter what value $k$ has.  Hence $t=0$, $q = 2k\binom{n}{2}$, and $s_2 = 12 k^2 \binom{n+1}{4}$ in Corollary \ref{T7}; these give the face and flat numbers of a generic planar $\cH(\Phi''_n;\bQ)$.  Thus $p_\cH(\lambda) = \lambda^2 - 2k\binom{n}{2}\lambda + 12k^2\binom{n+1}{4}$ generically.  
In this example we can go further: since there are no balanced triangles we can compute the number of balanced quadrilaterals and use this to get the Whitney numbers needed for the face and flat numbers in $\bbE^3$.  Omitting the somewhat lengthy details, the Whitney numbers $w_{i3}$ are given by Lemma \ref{Lwh} with $F_{n-3} = k^3 (n^3-5n-12) \binom{n}{3}$ and $t' = 6 \big[16\binom{k+1}{3} + 3k\big] \binom{n}{4}$.

\begin{figure} \mylabel{Feven4}
\vbox to 7.5truein{}
\begin{center}
\includegraphics{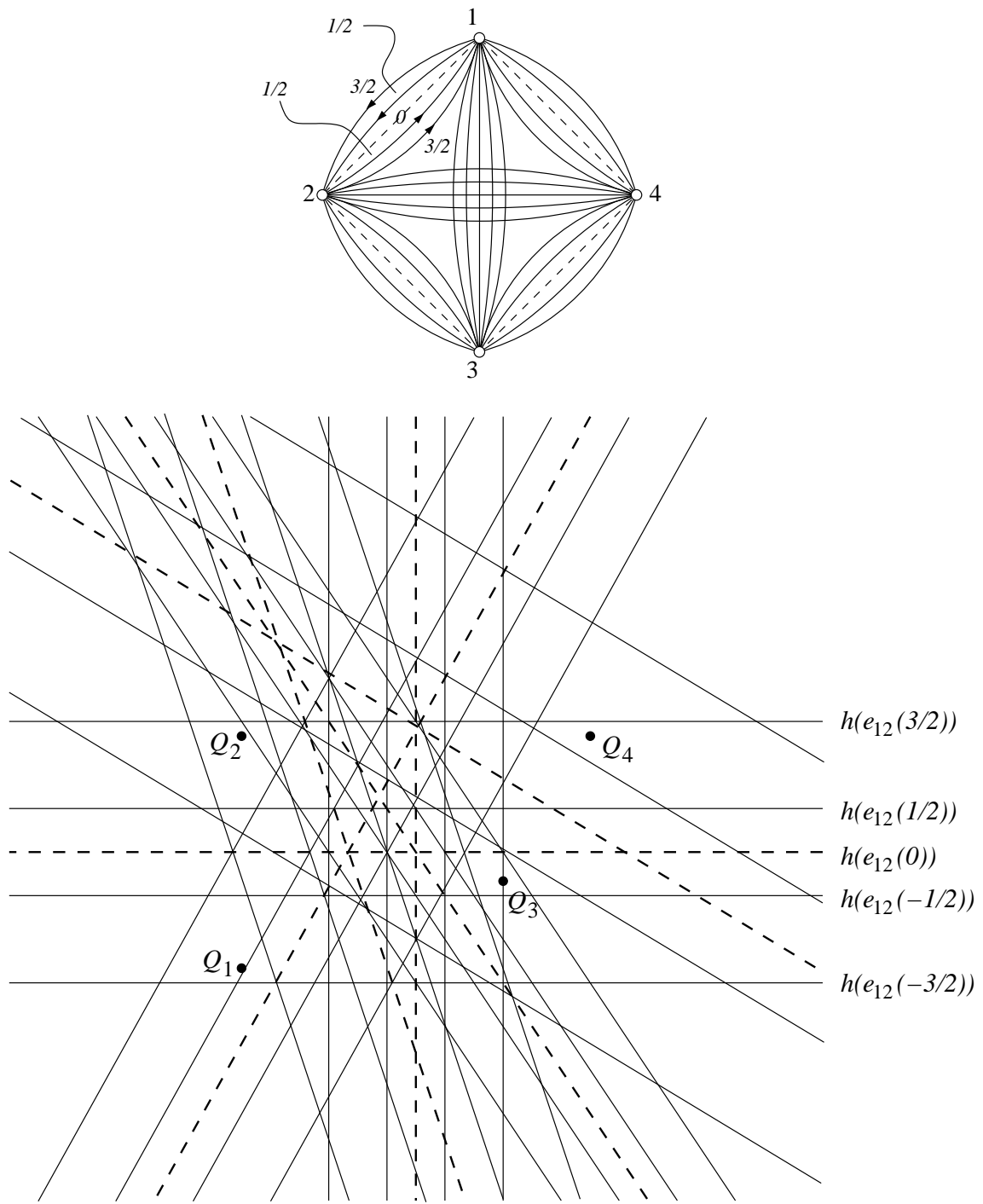}
\end{center}
\caption{The gain graphs and line arrangements of Example \ref{Xeven4}.  The heavy lines (which are dashed) are the perpendicular bisectors.}
\end{figure}

\begin{exam} \mylabel{Xeven4}
We take $k=2$ and $n=4$ so that we need $\Phi''_4 = \{\pm1,\pm3\}K_4$.  However, so that perpendiculars of each parallel class will be separated by the same distance as in Example \ref{Xodd4}, we divide the gains by 2, giving gain graph $\frac12\Phi''_4 = \{\pm\frac12,\pm\frac32\}K_4$.  This has no effect on balance, so the face and flat counts remain the same.  The solid lines in Figure 11.5 display the gain graph and corresponding generic Pythagorean arrangement.  Adding in the dashed lines we have $\frac12{\Phi''_4}^0 = \{0,\pm\frac12,\pm\frac32\}K_4$ and the associated Pythagorean arrangement.
\placefig{11.5}{Xeven4}
\end{exam}


\section{The voter debates} \mylabel{voter}

Once Good and Tideman had introduced the geometrical model of bisecting hyperplanes\footnote{The general idea of bisecting hyperplanes' modeling voter choice was well known.  As an application of Voronoi diagrams, it uses only parts of hyperplanes.  It was the new idea of Good and Tideman to consider the entire candidate ranking that gave a role to whole hyperplanes.  I am not aware of any other similar work.} 
it was easy to imagine more elaborate but also plausible variations and to ask for the number of possible outcomes.  In the more elaborate models, an outcome may no longer be a simple preference ranking.  For example, one of the variations is to allow discrete degrees of preference; an outcome is then the whole set of degrees of preference the voter feels between each pair of candidates.

In the Good--Tideman model, the voter must choose one over the other for every pair of candidates.  (We ignore borderline cases.)  A more
sophisticated voter would expect to see a measurable difference before
forming a preference.  We can grade models by the degree of sophistication
the voter shows in comparing two candidates.  A voter at the lowest level simply
chooses one of the two.  At the second level she has three options: prefer
the candidate who is significantly superior or remain neutral if neither
is.\footnote{The idea of an ``indifference region'' appeared (independently and somewhat earlier) in the nearest-neighbor study of Cacoullos \cite[\S 5]{Cac1}.  The delimiting hyperplanes of Cacoullos' indifference region are fixed by proportional distance, not Pythagorean coordinate, and are based on affinely independent reference points.  Still, this is the only place I know of that treats what we would call unbalanced gains.}  
A voter in the third level of sophistication distinguishes four
choices: strong or weak preference for either candidate.  One at the fourth
level adds the option of remaining neutral.  And so forth.

Each level of sophistication can lead to a variety of models depending on how the
voter chooses to define the boundaris between different options.  Let us assume the choice is always made on the basis of modified Pythagorean coordinates as in 
Equation \eqref{foot-num}: in other words, the regions
within which the voter chooses a particular option for evaluating $Q_i$ \emph{vis-{\'a}-vis} $Q_j$ are demarcated by hyperplanes perpendicular to the line $Q_iQ_j$, measured by Pythagorean coordinate times some power of the distance between the candidates.  (Otherwise we need a new theory.)  Then we have to 
consider how the voter might place the choice boundaries.  There are several
directions of classification.  First is the rule
of placement: the rule may specify for each hyperplane its Pythagorean
coordinate, its actual distance from the midpoint of segment $[Q_i,Q_j]$, its proportional distance as a fraction of $d(Q_i,Q_j)$, or some other exponent 
$\alpha$ in \eqref{foot-num}.  Second is
homogeneity: the same rule of location may be used for every pair of candidates, 
or the voter may emplace boundaries at random (random, of course, to the
observing sociometrician, who does not know the voter's reasons).  Third is symmetry: the boundaries for comparing
$Q_i$ with $Q_j$ may be symmetrical, or they may not.  Fourth is uniformity: the boundary hyperplanes for a given pair of candidates may be uniformly spaced (in the chosen measurement scheme) or they may not.

Thus there are five dimensions along which we can classify voting models
of our general type.  Fortunately for us, they lead to just a few mathematical situations.

Let us suppose the voter is at a definite level of sophistication, say the
$M$-th level, where there are $M+1$ options divided by $M$ hyperplanes
between each pair of candidates.

If the voter applies the distance or proportional rule of placement or any other 
rule of type \eqref{foot-num} with $\alpha \neq 0$, then (in general)
the only thing that matters is whether she chooses to use bisecting
hyperplanes.  (That is what we learn from Theorem \ref{T2}.)  If
she follows a symmetrical demarcation scheme with odd $M$, there will be a
bisector between each pair of candidates and we are in the situation of Section \ref{X3} (Section \ref{X1} if $M=1$).  Otherwise
she will not (in general) take any bisectors and we have the example of Section \ref{X2} with $m=M$.

If the voter uses a placement rule with random hyperplanes, the same remarks apply.  We will be in Section \ref{X2} with $m=M$ unless she deliberately chooses bisectors, when Section \ref{X3} or \ref{X1} applies.

A homogeneous Pythagorean placement rule leads to more varied mathematics.  

A uniform symmetrical rule puts us in Section \ref{X1} (Corollary \ref{T5}) if
$M=1$: this is the original case where there is one dividing hyperplane
for each pair and it is the bisector (since it is symmetrically placed).  When $M>1$, the voter is using dividing hyperplanes with coordinates
$\psi_{ij}(h) = 0, \pm \delta, \hdots, \pm l\delta$ when $M$ is odd, $M = 2l+1$, but
$\psi_{ij}(h) = \pm \delta, \pm 3\delta, \hdots, \pm (2l-1)\delta$ when $M = 2l$ is even.  This puts us in Section \ref{Xodd} if $M$ is odd and Section \ref{Xeven} if $M$ is even, Example \ref{Xoddnobi1} if $M=2$.

Asymmetric rules in general seem implausible, but there is one that deserves attention.  
That is where the voter always prefers one in each pair of candidates ($M=1$, the first level of sophistication) but has a biased
decision rule instead of simply choosing the nearer candidate.  Even an unsymmetrical
decision rule gives an arrangement of hyperplanes with balanced Pythagorean gain graph, 
if the voter has for each candidate $Q_i$ what we might call a ``prior bias'' $q_i$ 
(whether a preference or a prejudice) and
the dividing hyperplane $h_{ij}$ has Pythagorean coordinate $\psi_{ij} (h_{ij}) = q_i-q_j$; that is, it is offset due to the bias.  The more $q_j$ exceeds $q_i$, the more $h_{ij}$ will retreat from $Q_j$ in the direction of $Q_i$, so favoring $Q_j$.
Then Corollary \ref{T3} tells us that the voter has the same number of possible
rankings as in Good and Tideman's original model.

The actual preference rankings that correspond to the regions of $\cH$ are wholly unknown, although they must be somehow related to the oriented matroid structure of $\cH$.

\begin{resprob} \mylabel{RPorderings}
(a) What structure is there to the set $\cO$ of preference orderings that is realized by a generic arrangement of all the perpendicular bisectors of $n$ reference points in $d$-space?  ($\cO$ is not, as \eqref{E1} suggests, the set of all permutations of $\{1,\hdots,n\}$ that have at least $n-d$ cycles.  An easy example of 4 planar points suffices to disprove that.)  (b)  Which orderings correspond to unbounded regions?  (c)  Which new orderings appear as the dimension rises?  (d)  Characterize the sets $\cO$ that arise from all different possible generic reference points in fixed dimension.
\end{resprob}

\begin{resprob} \mylabel{RPpartial-orderings}
Generalize Research Problem \ref{RPorderings} to arrangements of only some perpendicular bisectors.
\end{resprob}

\begin{resprob} \mylabel{RPgen-orderings}
Generalize Research Problem \ref{RPorderings} to arrangements with indifference regions or degrees of preference.
\end{resprob}


\section{Further research} \mylabel{open}

\subsection{Nongenericity} \mylabel{nongeneric}

We observed in the introduction that an exact description of generic reference points is unknown.  This is unsatisfactory.

\begin{resprob} \mylabel{RPnongeneric}
Given a real, additive gain graph $\Phi$, characterize genericity (with respect to $\Phi$) of reference points.  That is, what are the properties of $\bQ$ that guarantee that $\cL(\cH(\Phi;\bQ))$ is generic?  At a minimum, characterize genericity amongst those $\bQ$ having ideal general position.
\end{resprob}

\subsection{Special position} \mylabel{specpos}

Suppose we require certain subsets of the $Q_i$ to be affinely dependent but ask for genericity in other respects; or suppose we even prescribe certain parallelisms.  It should still be possible to deduce the number of resulting regions by abstract combinatorial means, similar to those of this paper but considerably more complex.  Probably, the essential information about the reference points $Q_i$ is specified by their affine dependencies and the configuration at infinity of the lines they determine.  Abstractly these would be described by their affine dependence matroid and a comap of that matroid to account for behavior at infinity.

\begin{resprob} \mylabel{RPspecpos-structure}
Develop the structure theory of $\cL(\cH(\Phi,\bQ))$ and $\cL(\cH_\bbP(\Phi,\bQ))$ for reference points that have (a) specified affine dependencies, (b) simple position but specified behavior at infinity, or (c) both specified affine dependencies and specified behavior at infinity, but are otherwise generic.  Especially, compare their (semi)lattice structures to those of $\Latb\Phi$ and $\Lat L_0(\Phi)$.
\end{resprob}

Part (a) will be solved by defining the complete lift matroid of a biased graph with a matroid given on the vertex set, which is a problem of great intrinsic interest.

\begin{resprob} \mylabel{RPspecpos-enum}
Develop the enumerative theory of perpendicular dissections whose reference points have affine or infinite special position or both, but are otherwise generic.
\end{resprob}

\subsection{Hyperbolic dissections} \mylabel{hyperbolic}

The way Good and Tideman argued for their arrangement of bisectors was this:  The voter prefers the nearer of two candidates.  The boundary between preference domains is the locus of points equidistant from both candidates.  That is the perpendicular bisector of the connecting line segment.

If we apply the same reasoning to a more sophisticated voter, it leads us to a hyperbolic rather than a hyperplanar dissection.  Suppose the voter has second-level sophistication with the rule that she prefers one candidate only if its nearness exceeds the other's by some threshold $\delta$.  That is, if the voter falls in the domain where  $|d(P, Q_j) - d(P,Q_i)| < \delta$, she is neutral.  If $d(P,Q_j) - d(P,Q_i) > \delta$, she prefers $Q_i$.  This mathematics gives us three preference domains separated by the hyperboloid of revolution $|d(P,Q_j) - d(P,Q_i)| = \delta$.  The theory of hyperplane dissections does not apply.  

Problems of dissection by entire curved subspaces are little studied.  The planar hyperbolic Dirichlet tessellations (that is, Voronoi diagrams) of \cite[\S 2]{AB} are closely related, though in their subject (that of generalized Voronoi diagrams) only parts of curves and surfaces are employed.  There is a partial theory of topological dissections \cite{CATD}, which are more complicated than linear dissections because one has to determine not only the semilattice of intersections but also the topology of each intersection and each region, or at least their Euler characteristics.  It does seem possible that the special properties of hyperboloids on shared foci may make it possible to solve some cases at least of the hyperbolic preference ranking problem.

\begin{resprob}  \mylabel{RPhyperbolic}
Solve the second-level hyperbolic voter ranking problem in the plane.  The problem is to determine the maximum number of possible preference rankings if the voter $P$ prefers the nearer of candidates $Q_i$ and $Q_j$ if $|d(P, Q_i) - d(P,Q_j)| > \delta_{ij}$ and is neutral otherwise.  The $\delta_{ij}$ are positive numbers that may be taken all equal if that helps the solution.
\end{resprob}

\begin{resprob}  \mylabel{RPhyperbolic-generic}
Decide whether the solution to Problem \ref{RPhyperbolic} is generic.  That is, is the set of choices of candidates for which the maximum is attained dense in $(\bbE^2)^n$?
\end{resprob}

\subsection{Other scalar fields} \mylabel{scalars}

Let us speculate about an inner product space over an arbitrary field $F$.  (All the opinions in this section are unverified wish and hope.)
Our results should largely apply to any ordered field.  Regions, faces, and the intersection semilattice are defined; the field is topologized so generic position exists; everything seems to work except for a possible difficulty with gains, explained below.  Over an unordered field there is an intersection lattice but there are no regions or faces.  The fundamental problem there is that of genericity.  Defining it would seem to call for a topology on $F$, but it would be very interesting to see a definition that avoids this, perhaps by basing genericity on a suitable subfield like $\bbQ$ or the $p$-adics.  

Finite fields would require an altogether different approach.

The appropriate gain group for an arbitrary inner product space might not be $F^+$.  Let $D$ be the set of values of $\langle v,v\rangle - \langle w,w\rangle$ for $v,w \in F^d$, where $\langle x,y\rangle$ is the inner product.  The gains of a Pythagorean gain graph over $F$ must be chosen in $D$.  As long as $D$ is an additive group (as for instance when $F = \bbQ$ or $\bbR$, so $D=F$; or when $F=\bbC$, so $D=\bbR$), the parts of our theory that involve switching, as required for cross-sections and induced arrangments (Section \ref{xsect}), should carry over.

Complex space is especially interesting.  Every coefficient of the characteristic polynomial of $\cH$ has meaning.  The integral cohomology of the complement of the arrangement is determined by the intersection semilattice $\cL(\cH)$ and the rank of the $i$th cohomology group equals $|w_{d-i}(\cL(\cH))|$, the magnitude of the $(d-i)$th Whitney number of the first kind, by the theorems of Orlik and Solomon (\cite{OS}, \cite[Ch.\ 5]{OT}).


\end{document}